\documentclass[10pt,reqno]{amsart}
\usepackage{latexsym,amsmath}
\usepackage{times}
\usepackage{amsfonts}
\usepackage{amssymb}
\usepackage{latexsym}
\usepackage{hyperref}

\usepackage{times}
\usepackage{bbm}

\usepackage[dvips]{epsfig}
\graphicspath{{./Fig_eps/}{./Eps/}}
\DeclareGraphicsExtensions{.ps,.eps}
\usepackage{color}
\usepackage{fullpage}

\newcommand\nix{\,\cdot\,}

\newcommand\nb{\partial^\omega}
\newcommand\nbh[3]{\partial^\omega\bc{#1,#2,#3}}
\newcommand\nbg[2]{\partial^\omega\bc{{#1},{#2}}}

\newcommand\bplp{\pi_{n,p,k}^{\mathrm{pr}}}
\newcommand\plp{\pi_{n,m,k}^{\mathrm{pr}}}
\newcommand\prcp{\pi_{n,m,k}^{\mathrm{rr}}}
\newcommand\prm{\plp}
\newcommand\rrm{\prcp}

\newcommand\atom{\delta}

\newcommand\thet{\vartheta}

\newcommand{\beq}{\begin{equation}} \newcommand{\eeq}{\end{equation}}

\newcommand\dd{{\mathrm d}}
\newcommand\ism{\cong}

\newcommand\G{\vec G}
\newcommand\T{\vec T}

\DeclareMathOperator{\corr}{bias}

\newcommand\gnp{G(n,p)}

\newcommand\gnm{\G(n,m)}
\newcommand\Gnm{\G}

\numberwithin{equation}{section}

\def\vec#1{\mbox{\boldmath$\displaystyle#1$}}

\newcommand{\Zkc}{Z_{k}}

\newcommand{\dc}{d_{k,\mathrm{cond}}}

\newcommand{\dk}{d_{k-\mathrm{col}}}

\newcommand\MPCPS{Mathematical Proceedings of the Cambridge Philosophical Society}
\newcommand\IPL{Information Processing Letters}
\newcommand\COMB{Combinatorica}

\DeclareMathOperator{\pr}{\mathbb P}

\newcommand\plSIGMA{\hat\SIGMA}
\newcommand\plG{\hat\G}

\newcommand\SIGMA{\vec\sigma}

\newcommand\TAU{\vec\tau}

\newtheorem{definition}{Definition}[section]
\newtheorem{claim}[definition]{Claim}

\newtheorem{remark}[definition]{Remark}
\newtheorem{theorem}[definition]{Theorem}
\newtheorem{lemma}[definition]{Lemma}
\newtheorem{proposition}[definition]{Proposition}
\newtheorem{corollary}[definition]{Corollary}

\newcommand\id{\mathrm{id}}

\newcommand\cA{\mathcal{A}}
\newcommand\cB{\mathcal{B}}
\newcommand\cC{\mathcal{C}}

\newcommand\cF{\mathcal{F}}
\newcommand\cG{\mathcal{G}}
\newcommand\cE{\mathcal{E}}
\newcommand\cU{\mathcal{U}}

\newcommand\cS{\mathcal{S}}

\newcommand\cK{\mathcal{K}}

\newcommand\cL{\mathcal{L}}

\newcommand\cP{\mathcal{P}}
\newcommand\cX{\mathcal{X}}
\newcommand\cY{\mathcal{Y}}

\def\cR{{\mathcal R}}
\def\cC{{\mathcal C}}
\def\cE{{\mathcal E}}

\newcommand\eps{\varepsilon}

\newcommand\Erw{\mathbb{E}}

\newcommand{\vecone}{\vec{1}}

\newcommand{\Po}{{\rm Po}}

\newcommand{\Be}{{\rm Be}}

\newcommand\TV[1]{\left\|{#1}\right\|_{\mathrm{TV}}}

\newcommand{\bink}[2] {{{#1}\choose {#2}}}

\newcommand\ra{\rightarrow}

\newcommand\bc[1]{\left({#1}\right)}
\newcommand\cbc[1]{\left\{{#1}\right\}}
\newcommand\bcfr[2]{\bc{\frac{#1}{#2}}}
\newcommand{\bck}[1]{\left\langle{#1}\right\rangle}
\newcommand\brk[1]{\left\lbrack{#1}\right\rbrack}

\newcommand\norm[1]{\left\|{#1}\right\|}
\newcommand\abs[1]{\left|{#1}\right|}

\newcommand\RR{\mathbb{R}}

\newcommand{\whp}{w.h.p.}

\newcommand{\tensor}{\otimes}

\newcommand{\Erdos}{Erd\H{o}s}
\newcommand{\Renyi}{R\'enyi}

\newcommand\Lem{Lemma}
\newcommand\Prop{Proposition}
\newcommand\Thm{Theorem}
\newcommand\Cor{Corollary}
\newcommand\Sec{Section}

\begin{document}

\title{Local convergence of random graph colorings}

\author{Amin Coja-Oghlan$^*$, Charilaos Efthymiou$^{**}$ and Nor Jaafari}
\thanks{$^*$The research leading to these results has received funding from the European Research Council under the European Union's Seventh 
Framework Programme (FP/2007-2013) / ERC Grant Agreement n.\ 278857--PTCC}
\thanks{$^{**}$ Research is supported by ARC GaTech.}
\date{\today}

\address{Amin Coja-Oghlan, {\tt acoghlan@math.uni-frankfurt.de}, Goethe University, Mathematics Institute, 10 Robert Mayer St, Frankfurt 60325, Germany.}

\address{Charilaos Efthymiou, {\tt efthymiou@gmail.com}, Georgia Tech, College of Computing, 266 Ferst Drive,   Atlanta, 30332, USA.}

\address{Nor Jaafari, {\tt jaafari@math.uni-frankfurt.de}, Goethe University, Mathematics Institute, 10 Robert Mayer St, Frankfurt 60325, Germany.}

\maketitle

\begin{abstract}
\noindent
Let $\G=\gnm$ be a random graph whose average degree $d=2m/n$ is below the $k$-colorability threshold.
If we sample a $k$-coloring $\SIGMA$ of $\G$ uniformly at random, what can we say about the correlations between the colors assigned to
vertices that are far apart?
According to a prediction from statistical physics, for average degrees below the so-called {\em condensation threshold}  $\dc$,
the colors assigned to far away vertices are asymptotically independent [Krzakala et al.: Proc.~National Academy of Sciences 2007].
We prove this conjecture for $k$ exceeding a certain constant $k_0$.
More generally, we investigate the joint distribution of the $k$-colorings that $\SIGMA$ induces locally on the bounded-depth
neighborhoods of any fixed number of vertices.
In addition, we point out an implication on the \emph{reconstruction problem}.

\end{abstract}

\section{Introduction and results}

\noindent{\em
Let $\G=\gnm$ denote the random graph on the vertex set $\brk n=\cbc{1,\ldots,n}$ with precisely $m$ edges.
Unless specified otherwise, we assume that $m=m(n)=\lceil dn/2\rceil$ for a fixed number $d>0$.
As usual, $\gnm$ has a property $\cA$ ``with high probability'' (``\whp'') if $\lim_{n\to\infty}\pr\brk{\gnm\in\cA}=1$.
}

\subsection{Background and motivation} 
Going back to the seminal paper of \Erdos\ and \Renyi~\cite{ER} that founded the theory of random graphs,
the problem of coloring $\gnm$ remains one
of the longest-standing challenges in probabilistic combinatorics.
Over the past half-century, efforts have been devoted to determining the likely value of the chromatic number
 $\chi(\gnm)$~\cite{AchNaor,BBColor,LuczakColor,Matula} and its concentration~\cite{AlonKriv,Luczak,ShamirSpencer}
as well as to algorithmic problems such as  constructing or sampling colorings of the random graph~\cite{AchMolloy,DyerFr10,Efthy14,Efthy12,GMcD,KSud}.

A tantalising feature of the random graph coloring problem is the interplay between local and global effects.
{\em Locally} around almost any vertex the random graph is bipartite \whp\
In fact, for any fixed average degree $d>0$ and for any fixed $\omega$ the depth-$\omega$ neighborhood
of all but $o(n)$ vertices is just a tree \whp\
Yet {\em globally} the chromatic number of the random graph may be large.
Indeed, for any number $k\geq3$ of colors there exists a {\em sharp threshold sequence} $\dk=\dk(n)$
such that  for any fixed $\eps>0$, $\gnm$ is $k$-colorable \whp\ if $2m/n<\dk(n)-\eps$,
whereas the random graphs fails to be $k$-colorable \whp\ if $2m/n>\dk(n)+\eps$~\cite{AchFried}.
Whilst the thresholds $\dk$ are not known precisely, there are close upper and lower bounds.
The best current ones read
	\begin{equation}\label{eqdk}
	\dc=(2k-1)\ln k-2\ln 2+\delta_k\leq\liminf_{n\to\infty}\dk(n)\leq \limsup_{n\to\infty}\dk(n)\leq(2k-1)\ln k-1+\eps_k,
	\end{equation}
where  $\lim_{k\to\infty}\delta_k=\lim_{k\to\infty}\eps_k=0$~\cite{AchNaor,Covers,Danny}.
To be precise,
the lower bound in~(\ref{eqdk})
is formally defined as
	\begin{equation}\label{eqdc}
	\dc=\inf\cbc{d>0:\limsup_{n\to\infty}\Erw[\Zkc(\gnm)^{1/n}]<k(1-1/k)^{d/2}}.
	\end{equation}
This number, called the {\em condensation threshold} due to a connection with statistical physics~\cite{pnas},
can be computed precisely for $k$ exceeding a certain constant $k_0$~\cite{Cond}.
An asymptotic expansion yields the expression in~(\ref{eqdk}).

The contrast between local and global effects was famously pointed out by \Erdos,
who produced $\gnm$ as an example of a graph that simultaneously has a high chromatic number and a high girth~\cite{Erdos}.
The present paper aims at a more precise understanding of this collusion between short-range and long-range effects.
For instance, do global effects entail ``invisible'' constraints  on the colorings of the local neighborhoods so that certain ``local'' colorings do not
extend to a coloring of the entire graph?
And what correlations do typically exist between the colors of vertices at a large distance?

A natural way of formalising these questions is as follows.
Let $k\geq3$ be a number of colors, fix some number $\omega>0$
 and assume that $d<\dc$ so that  $\G=\gnm$ is $k$-colorable \whp\
Moreover, pick a vertex $v_0$ and fix a $k$-coloring $\sigma_0$ of its depth-$\omega$ neighborhood.
How many ways are there to extend $\sigma_0$ to a $k$-coloring of the entire graph, and how does this number depend on $\sigma_0$?
Additionally, if we pick a vertex $v_1$ that is ``far away'' from $v_0$ and if we pick another $k$-coloring $\sigma_1$ of the
depth-$\omega$ neighborhood of $v_1$, is there a $k$-coloring $\sigma$ of $\G$ that simultaneously extends both $\sigma_0$ and $\sigma_1$?
If so, how many such $\sigma$ exist, and how does this depend on $\sigma_0,\sigma_1$?

The main result of this paper (\Thm~\ref{Thm_xlwc} below) provides a very neat and accurate answer to these questions.
It shows that \whp\ all ``local'' $k$-colorings $\sigma_0$ extend
to {\em asymptotically  the same} number of $k$-colorings of the entire graph.
Let us write $\cS_k(G)$ for the set
of all $k$-colorings of a graph $G$ and
let $\Zkc(G)=|\cS_k(G)|$ be the number of $k$-colorings.
Moreover, let $\partial^\omega(G,v_0)$ be the depth-$\omega$ neighborhood of a vertex $v_0$ in $G$
	(i.e., the subgraph of $G$ obtained by deleting all vertices at distance greater than $\omega$ from $v_0$).
Then \whp\ any $k$-coloring $\sigma_0$ of $\partial^\omega(\G,v_0)$ has $$\frac{(1+o(1))\Zkc(\G)}{\Zkc(\partial^\omega(\G,v_0))}$$ extensions
to a $k$-coloring of $\G$.
Moreover, if we pick another vertex $v_1$ at random and fix some $k$-coloring $\sigma_1$ of the depth-$\omega$ neighborhood of $v_1$,
then \whp\ the number of joint extensions of $\sigma_0,\sigma_1$ is 
	$$\frac{(1+o(1))\Zkc(\G)}{\Zkc(\partial^\omega(\G,v_0))\Zkc(\partial^\omega(\G,v_1))}.$$
In other words, if we choose a $k$-coloring $\SIGMA$ uniformly at random, then the distribution of the $k$-coloring
that $\SIGMA$ induces on the subgraph $\partial^\omega(\G,v_0)\cup\partial^\omega(\G,v_1)$, which is a forest \whp, is asymptotically uniform.
The same statement extends to any fixed number $v_0,\ldots,v_l$ of vertices.

\subsection{Results}
The appropriate formalism for describing the limiting behavior of the local structure of the random graph is the
concept of \emph{local weak convergence}~\cite{Aldous,BenjaminiSchramm}.
The concrete instalment of the formalism that we employ is reminiscent of that used in~\cite{BST,MMS}.
(\Cor~\ref{Cor_xlwc} below provides a statement that is equivalent to the main result but that avoids the formalism of local weak convergence.)

Let $\mathfrak G$ be the set of all locally finite connected graphs whose vertex set is a countable subset of $\RR$.
Further, let $\mathfrak G_k$ be the set of all triples $(G,v_0,\sigma)$ such that $G\in\mathfrak G$, $\sigma:V(G)\to[k]$ is a $k$-coloring of $G$ and $v_0\in V(G)$ is a
distinguished vertex that we call the {\em root}.
We refer to  $(G,v_0,\sigma)$ as a {\em rooted $k$-colored graph}.
If $(G',v_0',\sigma')$ is another rooted $k$-colored graph, 
we call $(G,v_0,\sigma)$ and $(G',v_0',\sigma')$ {\em isomorphic} ($(G,v_0,\sigma)\ism(G',v_0',\sigma')$)
if there is an isomorphism $\varphi:G\to G'$ such that $\varphi(v_0)=\varphi(v_0')$, $\sigma=\sigma'\circ\varphi$
and such that for any $v,w\in V(G)$ such that $v<w$ we have $\varphi(v)<\varphi(w)$.
Thus, $\varphi$ preserves the root, the coloring and the order of the vertices (which are reals).
Let $[G,v_0,\sigma]$ be the isomorphism class of $(G,v_0,\sigma)$ and
let $\cG_k$ be the set of all isomorphism classes of rooted $k$-colored graphs.

For an integer $\omega\geq0$ and $\Gamma\in\cG_k$
we let $\partial^\omega\Gamma$ denote the isomorphism class of the rooted  $k$-colored graph obtained from $\Gamma$ by 
deleting all vertices whose distance from the root exceeds $\omega$.
Then any $\Gamma$, $\omega\geq0$ give rise to a function
	\begin{equation}\label{eqtopology}
	\cG_k\to\cbc{0,1},\qquad\Gamma'\mapsto\vecone\cbc{\partial^\omega\Gamma'= \partial^\omega\Gamma}.
	\end{equation}
We endow $\cG_k$ with the coarsest topology that makes all of these functions continuous.
Further, for  $l\geq1$ we equip $\cG_k^l$ with the corresponding product topology.
Additionally, the set $\cP(\cG_k^l)$ of probability measures on $\cG_k^l$ carries the weak topology,
as does the set $\cP^2(\cG_k^l)$ of all probability measures on $\cP(\cG_k^l)$.
The spaces $\cG_k^l,\cP(\cG_k^l),\cP^2(\cG_k^l)$ are Polish~\cite{Aldous}. 
For $\Gamma\in\cG_k$ we denote by $\atom_{\Gamma}\in\cP(\cG_k)$ the Dirac measure that puts mass one on  $\Gamma$.

Let $G$ be a finite $k$-colorable graph whose vertex set $V(G)$ is contained in $\RR$ and let $v_1,\ldots,v_l\in V(G)$.
Then we can define a probability measure on $\cG_k^l$ as follows.
Letting $ G \| v$ denote the connected component of $v\in V(G)$  and $\sigma \| v$ the restriction of $\sigma:V(G)\to[k]$ to $G\|v$, we define
	\begin{equation}\label{eqempirical}
	\lambda\bc{G,v_1,\ldots,v_l}=\frac1{\Zkc(G)}\sum_{\sigma\in \cS_k(G)}\bigotimes_{i=1}^l\atom_{[G\|v_i,v_i,\sigma\|v_i]}\in\cP(\cG_k^l).
	\end{equation}
The idea is that $\lambda_{G,v_1,\ldots,v_l}$ captures the joint empirical distribution
of colorings induced by a random coloring of  $G$ ``locally'' in the vicinity of the ``roots'' $v_1,\ldots,v_l$.
Further, let
	$$\vec\lambda_{n,m,k}^l=\frac1{n^l}\sum_{v_1,\ldots,v_l\in[n]}\Erw[\atom_{\lambda\bc{\gnm,v_1,\ldots,v_l}}|\chi(\gnm)\leq k]\in\cP^2(\cG_k^l).$$
This measure captures the typical distribution of the local colorings in a random graph with $l$ randomly chosen roots.
We are going to determine the limit of $\vec\lambda_{n,m,k}^l$ as $n\to\infty$.

To characterise this limit, let $\T^*(d)$ be a (possibly infinite) random Galton-Watson tree rooted at a vertex $v_0^*$ with offspring distribution $\Po(d)$.
We embed $\T^*(d)$ into $\RR$ by independently mapping each vertex to a uniformly random point in $[0,1]$;
	with probability one, all vertices get mapped to distinct points.
Let $\T(d)\in\mathfrak G$ signify the resulting random tree and let $v_0$ denote its root.
For a number $\omega>0$ we let $\partial^\omega\T(d)$
denote the (finite) rooted tree obtained from $\T(d)$ by removing all vertices at a distance greater than $\omega$ from $v_0$.
Moreover, for $l\geq1$ let $\T^{1}(d),\ldots,\T^{l}(d)$ be $l$ independent copies of $\T(d)$ and set
	\begin{align}\label{eqTreeSeq}
	\vec\thet_{d,k}^l\brk\omega&=\Erw\brk{\atom_{\bigotimes_{i\in[l]}\lambda\bc{\partial^\omega\T^{i}(d)}}}\in \cP^2(\cG_k^l),
		&\mbox{where}\\
	\lambda\bc{\partial^\omega\T^{i}(d)}&=\frac1{\Zkc(\partial^\omega\T^{i}(d))}\sum_{\sigma\in\cS_k(\partial^\omega\T^{i}(d))}
		\atom_{[\partial^\omega\T^{i}(d),v_0,\sigma]}\in\cP(\cG_k^l)&\mbox{ (cf.~(\ref{eqempirical})).}
			\nonumber
	\end{align}
The sequence $(\vec\thet_{d,k}^l\brk\omega)_{\omega\geq1}$ converges (see Appendix~\ref{Sec_TreeSeq}) and we let
	$$\vec\thet_{d,k}^l=\lim_{\omega\to\infty}\vec\thet_{d,k}^l\brk\omega.$$
Combinatorially, $\vec\thet_{d,k}^l$ 
corresponds to sampling $l$ copies of the Galton-Watson tree $\T(d)$ independently.
These trees are colored by assigning a random color to each of the $l$ roots independently and proceeding down each tree
by independently choosing a color for each vertex from the $k-1$
colors left unoccupied by the parent.

\begin{theorem}\label{Thm_xlwc}
There is a number $k_0>0$ such that for all $k\geq k_0$, $d<\dc$, $l>0$ we have
	$\lim_{n\to\infty}\vec\lambda_{n,m,k}^{l}=\vec\thet_{d,k}^{l}$.
\end{theorem}

Fix numbers $\omega\geq1$, $l\geq1$, choose a random graph $\G=\gnm$ for some large enough $n$  and choose vertices 
$\vec v_1,\ldots,\vec v_l$ uniformly and independently at random.
Then the depth-$\omega$ neighborhoods $\partial^\omega(\G,\vec v_1),\ldots,\partial^\omega(\G,\vec  v_l)$ 
are pairwise disjoint and the union $\cF=\partial^\omega(\G,\vec v_1)\cup\cdots\cup\partial^\omega(\G,\vec v_l)$ is a forest \whp\
Moreover, the distance between any two trees in $\cF$ is $\Omega(\ln n)$ \whp\
Given that $\G$ is $k$-colorable, let $\SIGMA$ be a random $k$-coloring of $\G$.
Then $\SIGMA$ induces a $k$-coloring of the forest $\cF$.
\Thm~\ref{Thm_xlwc} implies that \whp\ the distribution of the induced  coloring is at a total variation distance $o(1)$
from the uniform distribution on the set of all $k$-colorings of $\cF$.
Formally, let us write $\mu_{k,G}$ for the probability distribution on $[k]^{V(G)}$ defined by
	\begin{align*}
	\mu_{k,G}(\sigma)&=\vecone\cbc{\sigma\in\cS_k(G)}\Zkc(G)^{-1}&(\sigma\in[k]^{V(G)}),
	\end{align*}
i.e., the uniform distribution on the set of $k$-colorings of the graph $G$.
Moreover, for $U\subset V(G)$ let
	$\mu_{k,G|U}$ denote the projection of $\mu_{k,G}$ onto $[k]^U$, i.e.,
	\begin{align*}
	\mu_{k,G|U}(\sigma_0)&=\mu_{k,G}\bc{\cbc{\sigma\in[k]^V:\forall u\in U:\sigma(u)=\sigma_0(u)}}&(\sigma_0\in[k]^U).
	\end{align*}
If $H$ is a subgraph of $G$, then we just write $\mu_{k,G|H}$ instead of $\mu_{k,G|V(H)}$.
Let $\TV\nix$ denote the total variation norm.

\begin{corollary}\label{Cor_xlwc}
There is a constant $k_0>0$ such that
for any $k\geq k_0$, $d<\dc$, $l\geq1$, $\omega\geq0$ we have
	\begin{align*}
	\lim_{n\to\infty}\frac1{n^l}\sum_{v_1,\ldots,v_l\in[n]}
		\Erw\TV{\mu_{k,\G|\partial^\omega(\G,v_1)\cup\cdots\cup\partial^\omega(\G,v_l)}
			-\mu_{k,\partial^\omega(\G,v_1)\cup\cdots\cup\partial^\omega(\G,v_l)}}=0.
	\end{align*}
\end{corollary}

Since  \whp\ the pairwise distance of $l$ randomly chosen vertices $v_1,\ldots,v_l$ in $\G$ is $\Omega(\ln n)$,
	we observe that \whp
	$$\mu_{k,\partial^\omega(\G,v_1)\cup\cdots\cup\partial^\omega(\G,v_l)}=\bigotimes_{i\in[l]}\mu_{k,\partial^\omega(\G,v_i)}.$$
With very little work it can be verified that \Cor~\ref{Cor_xlwc} is actually equivalent to \Thm~\ref{Thm_xlwc}.
Setting $\omega=0$ in \Cor~\ref{Cor_xlwc} yields the following statement, which is of interest in its own right.

\begin{corollary}\label{Cor_decay}
There is a number $k_0>0$ such that for all $k\geq k_0$, $d<\dc$ and any integer $l>0$ we have
	\begin{equation}\label{eqCor_decay}
	\lim_{n\to\infty}\frac1{n^l}\sum_{v_1,\ldots,v_l\in[n]}
			\Erw\TV{\mu_{k,\G|\cbc{v_1,\ldots,v_l}}-\bigotimes_{i\in\brk l}\mu_{k,\G|\cbc{v_i}}}=0.
	\end{equation}
\end{corollary}

By the symmetry of the colors, $\mu_{k,\G|\cbc v}$ is just the uniform distribution on $\brk k$ for every vertex $v$.
Hence, \Cor~\ref{Cor_decay} states that for $d<\dc$ \whp\ in the random graph $\G$ for randomly chosen
vertices $\vec v_1,\ldots,\vec v_l$ the following is true:
if we choose a $k$-coloring $\SIGMA$ of $\G$ at random, then  $(\SIGMA(\vec v_1),\ldots,\SIGMA(\vec v_l))\in[k]^l$
is asymptotically uniformly distributed.
Prior results of Montanari and Gershenfeld~\cite{GM} and of Montanari, Restrepo and Tetali~\cite{Prasad}
imply that (\ref{eqCor_decay}) holds for $d<2(k-1)\ln(k-1)$, about an additive $\ln k$ below $\dc$.

The above results and their proofs are inspired by ideas from statistical physics.
More specifically, physicists have developed a 
non-rigorous but analytic technique, the so-called ``cavity method''~\cite{MM}, which has led to various conjectures on the random graph coloring problem.
These include a prediction as to the precise value of $\dc$ for any $k\geq3$~\cite{LenkaFlorent} as well as
a conjecture as to the precise value of the $k$-colorability threshold $\dk$~\cite{KPW}.
While the latter formula is  complicated, asymptotically we expect that $\dk=(2k-1)\ln k-1+\eps_k$, where $\lim_{k\to\infty}\eps_k=0$.
According to this conjecture, the upper bound in~(\ref{eqdk}) is asymptotically tight and $\dk$ is strictly greater than $\dc$.
Furthermore, according to the physics considerations (\ref{eqCor_decay}) holds for any $k\geq3$ and any $d<\dc$~\cite{pnas}.
\Cor~\ref{Cor_decay} verifies this conjecture for $k\geq k_0$.
By contrast, according to the physics predictions, (\ref{eqCor_decay}) does {\em not} hold for $\dc<d<\dk$.
As (\ref{eqCor_decay}) is the special case of $\omega=0$ of \Thm~\ref{Thm_xlwc} (resp.\ \Cor~\ref{Cor_xlwc}), the conjecture implies
that neither of these extend to $d>\dc$.
In other words, the physics picture suggests that \Thm~\ref{Thm_xlwc}, \Cor~\ref{Cor_xlwc} and \Cor~\ref{Cor_decay}
are \emph{optimal}, except that the assumption $k\geq k_0$ can possibly be replaced by $k\geq3$.

\subsection{An application}
Suppose we draw a $k$-coloring $\SIGMA$ of $\G$ at random.
Of course, the colors that $\SIGMA$ assigns to the neighbors of a vertex $v$ and the color of $v$ are correlated
	(they must be distinct).
More generally, it seems reasonable to expect that for any {\em fixed} ``radius'' $\omega$ the colors assigned to the
vertices at distance $\omega$ from $v$ and the color of $v$ itself will typically be correlated.
But will these correlations persist as $\omega\to\infty$?
This is the  ``reconstruction problem'', which has received considerable attention in the context
of random constraint satisfaction problems in general and in random graph coloring in particular~\cite{pnas,Prasad,SlyReconstruction}.
To illustrate the use of \Thm~\ref{Thm_xlwc} we will show how it readily implies the result on the reconstruction
problem for random graph coloring from~\cite{Prasad}.

To formally state the problem, assume that $G$ is a finite $k$-colorable graph.
For $v\in V(G)$ and a subset $\emptyset\neq\cR\subset\cS_k(G)$ let $\mu_{k,G|v}(\nix|\cU)$ be the probability distribution on $\brk k$ defined by
	\begin{align*}
	\mu_{k,G|v}(i|\cR)&=\frac1{\abs\cR}\sum_{\sigma\in\cR}\vecone\cbc{\sigma(v)=i},
	\end{align*}
i.e., the distribution of the color of $v$ in a random coloring $\sigma\in\cR$.
For $v\in V(G)$, $\omega\geq1$ and $\sigma_0\in\cS_k(G)$ let
	$$\cR_{k,G}(v,\omega,\sigma_0)=\cbc{\sigma\in\cS_k(G):\forall u\in V(G)\setminus\partial^{\omega-1}(G,v):\sigma(u)=\sigma_0(u)}.$$
Thus, $\cR_{k,G}(v,\omega,\sigma_0)$ contains all $k$-colorings that coincide with $\sigma_0$ on vertices whose distance from $v$ is 
	{\em at least} $\omega$.
Moreover, let
	\begin{align*}
	\corr_{k,G}(v,\omega,\sigma_0)&=\frac12\sum_{i\in [k]}\left| \mu_{k,G|v}(i|\cR_{k,G}(v,\omega,\sigma_0))
		-\frac1k\right |, &
			\corr_{k,G}(v,\omega)&=\frac1{\Zkc(G)}\sum_{\sigma_0\in\cS_k(G)}\corr_{k,G}(v,\omega,\sigma_0).
	\end{align*}
Clearly, for symmetry reasons, if we draw a $k$-coloring $\SIGMA\in\cS_k(G)$ uniformly at random,
then $\SIGMA(v)$ is uniformly distributed over $\brk k$.
What $\corr_{k,G}(v,\omega,\sigma_0)$ measures is how much conditioning on the event  $\SIGMA\in\cR_{k,G}(v,\omega,\sigma_0)$ biases the color of $v$.
Accordingly, $\corr_{k,G}(v,\omega)$  measures the bias induced by a {\em random} ``boundary condition'' $\sigma_0$.
We say that \emph{non-reconstruction} occurs in $\gnm$ if 
	$$\lim_{\omega\to\infty}\lim_{n\to\infty}\frac1n\sum_{v\in[n]}\Erw[\corr_{k,\gnm}(v, \omega)]=0.$$
Otherwise, \emph{reconstruction} occurs.
Analogously, 
recalling that $\vec T(d)$ is the Galton-Watson tree rooted at $v_0$, we say that
\emph{tree non-reconstruction} occurs at $d$ if
	$\lim_{\omega\to\infty}\Erw[\corr_{k,\partial^\omega\vec T(d)}(v_0, \omega )]=0.$
Otherwise, \emph{tree reconstruction} occurs.

\begin{corollary}\label{Cor_reconstr}
There is a number $k_0>0$ such that for all $k\geq k_0$ and  $d<\dc$ the following is true.
	\begin{equation}\label{eqCor_reconstr}
	\parbox{12cm}{Reconstruction occurs in $\gnm$ $\Leftrightarrow$ tree reconstruction occurs at $d$.}
	\end{equation}
\end{corollary}

Montanari, Restrepo and Tetali~\cite{Prasad} proved~(\ref{eqCor_reconstr}) for $d<2(k-1)\ln(k-1)$, about an additive $\ln k$ below $\dc$.
This gap could be plugged by invoking recent results on the geometry of the set of $k$-colorings~\cite{Silent,Covers,Molloy}.
However, we shall see that \Cor~\ref{Cor_reconstr} is actually an immediate consequence of \Thm~\ref{Thm_xlwc}.

The point of \Cor~\ref{Cor_reconstr} is that it reduces the reconstruction problem on a combinatorially extremely 
intricate  object, namely the random graph $\gnm$, to the same problem on a much simpler  structure,
 	namely the Galton-Watson tree $\T(d)$.
That said, the reconstruction problem on $\T(d)$ is far from trivial.
The best current bounds show that there exists a sequence $(\delta_k)_k\to 0$ such that non-reconstruction holds in $\T(d)$ if $d<(1-\delta_k)k\ln k$ while 
reconstruction occurs if $d>(1+\delta_k)k\ln k$~\cite{GWReconstruction}.

\subsection{Techniques and outline}
None of the arguments in the present paper are particularly difficult.
It is rather that a combination of several relatively simple ingredients proves remarkably powerful.
The starting point of the proof is a recent result~\cite{Silent} on the concentration of the number $\Zkc(\gnm)$ of $k$-colorings of $\gnm$.
This result entails a very precise connection between a fairly simple probability distribution, the so-called ``planted model'', and the experiment
of sampling a random coloring of a random graph, thereby extending the ``planting trick'' from~\cite{Barriers}.
However, this planting argument is not powerful enough to establish \Thm~\ref{Thm_xlwc} (cf.\ also the discussion in~\cite{BST}).
Therefore, in the present paper the key idea is to use the information about $\Zkc(\gnm)$ to introduce an enhanced variant of the planting trick.
More specifically, in \Sec~\ref{Sec_planting} we will establish a connection between the experiment of sampling a random {\em pair} of colorings of $\gnm$
and another, much simpler probability distribution that we call the {\em planted replica model}.
We expect that this idea will find future uses.

Apart from the concentration of $\Zkc(\gnm)$, this connection also hinges on a study of the ``overlap'' of two randomly chosen colorings of $\gnm$.
The overlap was studied in prior work on reconstruction~\cite{GM,Prasad} in the case that $d<2(k-1)\ln(k-1)$ based on considerations
from the second moment argument of Achlioptas and Naor~\cite{AchNaor} that gave the best lower bound on the $k$-colorability threshold at the time.
To extend the study of the overlap to the whole range $d\in(0,\dc)$, we crucially harness insights from the improved second moment
argument from~\cite{Danny} and the rigorous derivation of the condensation threshold~\cite{Cond}.

As we will see in \Sec~\ref{Sec_Nor},
the study of the planted replica model allows us to draw conclusions as to the typical ``local'' structure of pairs of random colorings of $\gnm$.
To turn these insights into a proof of \Thm~\ref{Thm_xlwc}, in \Sec~\ref{Sec_lwc} we extend an elegant argument from~\cite{GM}, which was used there to
establish the asymptotic independence of the colors assigned to a bounded number of randomly chosen individual vertices (reminiscent of~(\ref{eqCor_decay}))
for $d<2(k-1)\ln(k-1)$.

The bottom line is that the strategy behind the proof of \Thm~\ref{Thm_xlwc} is rather generic.
It probably extends to other problems of a similar nature.
A natural class to think of are the binary problems studied in~\cite{Prasad}. 
Another candidate might be the hardcore model, which was studied in~\cite{BST} by a somewhat different approach.

\section{Preliminaries}

\subsection{Notation}
For a finite or countable set $\cX$ we denote by $\cP(\cX)$ the set of all probability distributions on $\cX$,
which we identify with the set of all maps $p:\cX\to[0,1]$ such that $\sum_{x\in\cX}p(x)=1$.
Furthermore, if $N>0$ is an integer, then $\cP_N(\cX)$ is the set of all $p\in\cP(\cX)$ such that $Np(x)$ is an integer for every $x\in\cX$.
With the convention that $0\ln0=0$, we denote the entropy of $p\in\cP(\cX)$ by
	$$H(p)=-\sum_{x\in\cX}p(x)\ln p(x).$$

Let $G$ be a $k$-colorable graph.
By $\SIGMA^{k,G},\SIGMA^{k,G}_1,\SIGMA^{k,G}_2,\ldots\in\cS_k(G)$ we denote independent uniform samples from $\cS_k(G)$.
Where $G,k$ are apparent from the context, we omit the superscript.
Moreover, if  $X:\cS_k(G)\to\RR$, we write
	$$\bck{X(\SIGMA)}_{G,k}=\frac1{\Zkc(G)}\sum_{\sigma\in\cS_k(G)}X(\sigma).$$
More generally, if $X:\cS_k(G)^l\to\RR$, then
	$$\bck{X(\SIGMA_1,\ldots,\SIGMA_l)}_{G,k}=\frac1{\Zkc(G)^l}\sum_{\sigma_1,\ldots,\sigma_l\in\cS_k(G)}X(\sigma_1,\ldots,\sigma_l).$$
We omit the subscript $G$ and/or $k$ where it is apparent from the context.

Thus, the symbol $\bck\nix_{G,k}$ refers to the average over randomly chosen $k$-colorings of a {\em fixed} graph $G$.
By contrast, the standard notation $\Erw\brk\nix$, $\pr\brk\nix$ will be used to indictate that the expectation/probability is taken
over the choice of the random graph $G(n,m)$.
Unless specified otherwise, we use the standard $O$-notation to refer to the limit $n\to\infty$.
Throughout the paper, we tacitly assume that $n$ is sufficiently large for our various estimates to hold.

By a {\em rooted graph} we mean a graph $G$ together with a distinguished vertex $v$, the {\em root}.
The vertex set is always assumed to be a subset of $\RR$.
If $\omega\geq0$ is an integer, then $\nbg{G}{v}$ signifies the subgraph of $G$ obtained by removing all
vertices at distance greater than $\omega$ from $v$ (including those vertices of $G$ that are not reachable from $v$), rooted at $v$.
An {\em isomorphism} between two rooted graphs $(G,v)$, $(G',v')$ is an isomorphism $G\to G'$ of the underlying graphs
that maps $v$ to $v'$ and that preserves the order of the vertices (which is why we insist that they be reals).

\subsection{The first moment}\label{Sec_firstMoment}
The present work builds upon results on the first two moments of $\Zkc(\gnm)$.

\begin{lemma}\label{Lemma_firstMoment}
For any $d>0$, 
$\Erw[\Zkc(\G)]=\Theta(k^n(1-1/k)^m).$
\end{lemma}

Although \Lem~\ref{Lemma_firstMoment} is folklore, we briefly comment on how the expression comes about.
For $\sigma:\brk n\to\brk k$ let
	\begin{equation}\label{eqForb1}
	\cF(\sigma)=\sum_{i=1}^k\bink{|\sigma^{-1}(i)|}{2}
	\end{equation}
be the number of edges of the complete graph that are monochromatic under $\sigma$.
Then
	\begin{align}\label{eqForb2}
	\pr\brk{\sigma\in\cS_k(\G)}&=\bink{\bink n2-\cF(\sigma)}{m}\bigg/\bink{\bink n2}{m}.
	\end{align}
By convexity, we have $\cF(\sigma)\geq\frac1k\bink n2$ for all $\sigma$.
In combination with~(\ref{eqForb2}) and the linearity of expectation, this implies that $\Erw[\Zkc(\gnm)]= O(k^n(1-1/k)^m).$
Conversely, there are $\Omega(k^n)$ maps $\sigma:\brk n\to\brk k$ such that $\left|n/k-|\sigma^{-1}(i)| \right|\leq\sqrt n$ for all $i$, and
$\cF(\sigma)/\bink n2=1/k+O(1/n)$ for all such $\sigma$.
This implies  $\Erw[\Zkc(\G)]=\Omega(k^n(1-1/k)^m).$
The following result shows that $\Zkc(\G)$ is tightly concentrated about its expectation for $d<\dc$.

\begin{theorem}[\cite{Silent}]\label{Thm_Z}
There is $k_0>0$ such that for all $k\geq k_0$ and all $d<\dc$ we have
	$$\lim_{\omega\to\infty}\lim_{n\to\infty}\pr\brk{|\ln Z_k(\G)-\ln\Erw[Z_k(\G)]|\leq\omega}=1.$$
\end{theorem}

For $\alpha=(\alpha_1,\ldots,\alpha_k)\in\cP_n([k])$ we let $Z_{\alpha}(\G)$ be the number of $k$-colorings $\sigma$ of $\G$
such that $|\sigma^{-1}(i)|=\alpha_i n$ for all $i\in[k]$.
Conversely, for a map $\sigma:\brk n\to\brk k$ let $\alpha(\sigma)=n^{-1}(\sigma^{-1}(i))_{i\in[k]}\in\cP_n(\brk k)$.
Additionally, let $\bar\alpha=k^{-1}\vecone=(1/k,\ldots,1/k)$.

\begin{lemma}[{\cite[\Lem~3.1]{Silent}}]\label{Lemma_phiFirstMoment}
Let
	$\varphi(\alpha)=H(\alpha)+\frac{d}2\ln\bc{1-\norm\alpha_2^2}$.
Then 
	\begin{align*}
	\Erw[Z_\alpha(\Gnm)]&=O(\exp(n\varphi(\alpha)))&\mbox{uniformly for all }\alpha\in\cP_n(\brk k),\\
	\Erw[Z_\alpha(\Gnm)]&=\Theta(n^{(1-k)/2})\exp(n\varphi(\alpha))&
		\mbox{uniformly for all $\alpha\in\cP_n(\brk k)$ such that $\norm{\alpha-\bar\alpha}_2\leq k^{-3}$}.
	\end{align*}
\end{lemma}

\subsection{The second moment}\label{Sec_secondMoment}
Define the {\em overlap} of $\sigma,\tau:[n]\to[k]$ as the $k\times k$ matrix $\rho(\sigma,\tau)$ with entries
	$$\rho_{ij}(\sigma,\tau)=\frac1n\abs{\sigma^{-1}(i)\cap\tau^{-1}(j)}.$$
Then the number of edges of the complete graph that are monochromatic under either $\sigma$ or $\tau$ equals
	$$\cF(\sigma,\tau)=\cF(\sigma)+\cF(\tau)-\sum_{i,j\in\brk k}\bink{n\rho_{ij}(\sigma,\tau)}{2}.$$
For $i\in\brk k$ let $\rho_{i\nix}$ signify the $i$th row of the matrix $\rho$, and for $j\in\brk k$ let $\rho_{\nix j}$ denote the $j$th column.
An elementary application of inclusion/exclusion yields (cf.~\cite[Fact~5.4]{Silent})
	\begin{align}\label{eqErwZrho}
	\pr[\sigma,\tau\in\cS_k(\Gnm)]&
		=\frac{\bink{\bink n2-\cF(\sigma,\tau)}m}{\bink{\bink n2}m}
		=O\bc{\brk{1-\sum_{i\in[k]}(\norm{\rho_{i\nix}(\sigma,\tau)}_2^2+\norm{\rho_{\nix i}(\sigma,\tau)}_2^2)
			+\norm{\rho(\sigma,\tau)}_2^2}^m}.
	\end{align}

We can view $\rho(\sigma,\tau)$ as a distribution on $\brk k\times\brk k$, i.e., $\rho(\sigma,\tau)\in\cP_n(\brk k^2)$.
Let $\bar\rho$ be the uniform distribution on $\brk k^2$.
Moreover, for $\rho\in\cP_{n}(\brk k^2)$ let $Z_\rho^\tensor(\Gnm)$ be the number of pairs $\sigma_1,\sigma_2\in\cS_k(\Gnm)$ with overlap $\rho$.
Finally, let
	\begin{align}
	\cR_{n,k}(\omega)&=\cbc{\rho\in\cP_n([k]^2):\forall i\in\brk k:
	\norm{\rho_{i\nix}-\bar\alpha}_2,\norm{\rho_{\nix i}-\bar\alpha}_2\leq\sqrt{\omega/n}},&\mbox{and}\\
	f(\rho)&=H(\rho)+\frac d2\ln(1-2/k+\norm{\rho}_2^2).
	\end{align}

\begin{lemma}[\cite{AchNaor}]\label{Lemma_f}
Assume that $\omega=\omega(n)\to\infty$ but $\omega=o(n)$.
For all $k\geq3,d>0$ we have
	\begin{align*}
	\Erw[Z_\rho^\tensor(\Gnm)]&=O(n^{(1-k^2)/2})\exp(nf(\rho))&
		\mbox{uniformly for all $\rho\in\cR_{n,k}(\omega)$ s.t.\ }\norm{\rho-\bar\rho}_\infty\leq k^{-3},\\
	\Erw[Z_\rho^\tensor(\Gnm)]&=O(\exp(nf(\rho)))&\mbox{uniformly for all $\rho\in\cR_{n,k}(\omega)$.}
	\end{align*}
Moreover, if $d<2(k-1)\ln(k-1)$, then for any $\eta>0$ there exists $\delta>0$ such that
	\begin{equation}\label{eqAchNaor}
	f(\rho)<f(\bar\rho)-\delta\qquad\mbox{ for all $\rho\in\cR_{n,k}(\omega)$ such that $\norm{\rho-\bar\rho}_2>\eta$}.
	\end{equation}
\end{lemma}

The bound~(\ref{eqAchNaor}) applies for $d<2(k-1)\ln(k-1)$, about $\ln k$ below $\dc$.
To bridge the gap,  let $\kappa=1-\ln^{20}k/k$ and call $\rho\in\cP_n(\brk k^2)$ {\em separable} if $k\rho_{ij}\not\in(0.51,\kappa)$ for all $i,j\in[k]$.
Moreover, $\sigma\in\cS_k(\G)$ is {\em separable} if $\rho(\sigma,\tau)$ is separable for all $\tau\in\cS_k(\G)$.
Otherwise, we call $\sigma$ {\em inseparable}.
Further, $\rho$ is {\em $s$-stable} if there are precisely $s$ entries such that $k\rho_{ij}\geq\kappa$.

\begin{lemma}[\cite{Danny}]\label{Lemma_Danny}
There is $k_0$ such that for all $k>k_0$ and all
$2(k-1)\ln(k-1)\leq d\leq2k\ln k$ the following is true.
\begin{enumerate}
\item Let $\tilde Z_k(\G)=\abs{\cbc{\sigma\in\cS_k(\G):\sigma\mbox{ is inseparable}}}$.
	Then
		$\Erw[\tilde Z_k(\G)]\leq\exp(-\Omega(n))\Erw[\Zkc(\G)]$.
\item Let $1\leq s\leq k-1$.
	Then $f(\rho)<f(\bar\rho)-\Omega(1)$ uniformly for all $s$-stable $\rho$.
\item For any $\eta>0$ there is $\delta>0$ such that
		$\sup\{f(\rho):\mbox{$\rho$ is $0$-stable and $\norm{\rho-\bar\rho}_2>\eta$}\}<f(\bar\rho)-\delta$.
		\end{enumerate}
\end{lemma}

\Lem~\ref{Lemma_Danny} omits the $k$-stable case.
To deal with it, we introduce
	\begin{equation}\label{eqcluster}
	\cC(G,\sigma)=\cbc{\tau\in\cS_k(G):\rho(\sigma,\tau)\mbox{ is $k$-stable}}.
	\end{equation}

\begin{lemma}[\cite{Cond}]\label{Lemma_clusterSize}
There exist $k_0$ and $\omega=\omega(n)\to\infty$ such that for all $k\geq k_0$, $2(k-1)\ln(k-1)\leq d<\dc$ we have
	\begin{align*}
	\lim_{n\to\infty}\pr\brk{\bck{|\cC(\G,\SIGMA)|}_{\G,k}\leq\omega^{-1}\Erw\brk{\Zkc(\G)}}=1.
	\end{align*}
\end{lemma}

\subsection{A tail bound}
Finally, we need the following inequality.

\begin{lemma}[\cite{Lutz}]\label{Lemma_Lutz}
Let $X_1,\ldots,X_N$ be independent random variables with values in a finite set $\Lambda$.
Assume that $f:\Lambda^N\ra\RR$ is a function, that $\Gamma\subset\Lambda^N$ is an event and that $c,c'>0$ are numbers such that 
the following is true.
	\begin{equation}\label{eqTL}
	\parbox{12cm}{If $x,x'\in\Lambda^N$ are such that there is $k\in\brk N$ such that $x_i=x_i'$ for all $i\neq k$, then
			$$|f(x)-f(x')|\leq\left\{\begin{array}{cl}
				c&\mbox{ if }x\in\Gamma,\\
				c'&\mbox{ if }x\not\in\Gamma.
				\end{array}\right.$$}
	\end{equation}
Then for any $\gamma\in(0,1]$
and any $t>0$ we have
	$$\pr\brk{|f(X_1,\ldots,X_N)-\Erw[f(X_1,\ldots,X_N)]|>t}
		\leq2\exp\bc{-\frac{t^2}{2N(c+\gamma (c'-c))^2}}+\frac{2 N}\gamma\pr\brk{(X_1,\ldots,X_N)\not\in\Gamma}.$$
\end{lemma}

\section{The planted replica model}\label{Sec_planting}

\noindent{\em Throughout this section we assume that $k\geq k_0$ for some large enough constant $k_0$ and that $d<\dc$.}

\medskip
\noindent
In this section we introduce the key tool for the proof of \Thm~\ref{Thm_xlwc}, the {\em planted  replica model}.
This is the probability distribution $\plp$ on triples $(G,\sigma_1,\sigma_2)$ such that $G$ is a graph on $[n]$ with $m$ edges
and $\sigma_1,\sigma_2\in\cS_k(G)$ induced by the following experiment.

\begin{description}
\item[PR1] Sample two maps $\plSIGMA_1,\plSIGMA_2:[n]\to[k]$ independently and uniformly at random subject to the condition
		that $\cF(\plSIGMA_1,\plSIGMA_2)\leq\bink n2-m$.
\item[PR2] Choose a graph $\plG$ on $[n]$ with precisely $m$ edges uniformly at random, subject to the condition that
		both $\plSIGMA_1,\plSIGMA_2$ are proper $k$-colorings.
\end{description}

\noindent
We define
	$$\plp(G,\sigma_1,\sigma_2)=\pr\brk{(\plG,\plSIGMA_1,\plSIGMA_2)=(G,\sigma_1,\sigma_2)}.$$
Clearly, the planted replica model is quite tame so that it should be easy to bring the known techniques from the theory of random graphs to bear.
Indeed, the conditioning in {\bf PR1} is harmless because $\Erw[\cF(\plSIGMA_1,\plSIGMA_2)]\sim(2/k-1/k^2)\bink n2$ while $m=O(n)$.
Hence, by the Chernoff bound we have $\cF(\plSIGMA_1,\plSIGMA_2)\leq\bink n2-m$ \whp\
Moreover, {\bf PR2} just means that we draw $m$ random edges out of the $\bink n2-\cF(\plSIGMA_1,\plSIGMA_2)$ edges of the complete graph
that are bichromatic under both $\plSIGMA_1,\plSIGMA_2$.
In particular, we have the explicit formula
	\begin{align*}
	\plp(G,\sigma_1,\sigma_2)&=\frac1
			{\abs{\cbc{(\tau_1,\tau_2)\in[k]^n\times [k]^n:\cF(\tau_1,\tau_2)\leq\bink n2-m}}}
			{\sum_{\tau_1,\tau_2:\brk n\to\brk k,\,\cF(\tau_1,\tau_2)\leq\bink n2-m}\bink{\bink n2-\cF(\tau_1,\tau_2)}{m}^{-1}}.
	\end{align*}

The purpose of the planted replica model is to get a handle on another experiment, which at first glance seems far less amenable.
The {\em random replica model}  $\prcp$ is a probability distribution on triples $(G,\sigma_1,\sigma_2)$ such that $\sigma_1,\sigma_2\in\cS_k(G)$ as well.
It is induced by the following experiment.

\begin{description}
\item[RR1] Choose a random graph $\G=\gnm$ subject to the condition that $\G$ is $k$-colorable.
\item[RR2] Sample two colorings $\SIGMA_1,\SIGMA_2$ of $\G$ uniformly and independently.
\end{description}

\noindent
Thus,  the random replica model is defined by the formula
	\begin{align}\label{eqrr}
	\prcp(G,\sigma_1,\sigma_2)&=\pr\brk{(\G,\SIGMA_1,\SIGMA_2)=(G,\sigma_1,\sigma_2)}=
		\brk{\bink{\bink n2}m\pr\brk{\chi(\G)\leq k}\Zkc(G)^2}^{-1}.
	\end{align}
Since we assume that $d<\dc$, $\G$ is $k$-colorable \whp\
Hence, the conditioning in {\bf RR1} is innocent.
But this is far from true of the experiment described in {\bf RR2}.
For instance, we have no idea as to how one might implement {\bf RR2} constructively for $d$ anywhere near $\dc$.
In fact, the best current algorithms for finding a single $k$-coloring of $\G$, let alone a random pair, stop working for degrees $d$ about a factor of two
	below $\dc$ (cf.\ \cite{Barriers}).

Yet the main result of this section shows that for $d<\dc$, the ``difficult'' random replica model can be studied by means of the ``simple'' planted replica model.
More precisely, recall that a sequence $(\mu_n)_{n}$  of probability measures is {\em contiguous} with respect to another sequence $(\nu_n)_n$
if $\mu_n,\nu_n$ are defined on the same ground set for all $n$ and if for any sequence $(\cA_n)_n$ of events such that $\lim_{n\to\infty}\nu_n(\cA_n)=0$
we have $\lim_{n\to\infty}\mu_n(\cA_n)=0$.

\begin{proposition}\label{Thm_contig}
If $d<\dc$, then $\rrm$ is contiguous with respect to $\prm$.
\end{proposition}

The rest of this section is devoted to the proof of \Prop~\ref{Thm_contig}.
A key step is to study the distribution of the overlap of two random $k$-colorings $\SIGMA_1,\SIGMA_2$ of $\G$,
whose definition we recall from \Sec~\ref{Sec_secondMoment}.

\begin{lemma}
\label{Lemma_Z2}
Assume that $d<\dc$.
Then $\Erw[\bck{\norm{\rho(\SIGMA_1,\SIGMA_2)-\bar\rho}_2}_{\G}]=o(1).$
\end{lemma}

In words, \Lem~\ref{Lemma_Z2} asserts that the expectation over the choice of the random graph $\G$
	(the outer $\Erw$) of the average $\ell_2$-distance of the overlap of two randomly chosen $k$-colorings of $\G$
	from $\bar\rho$ goes to $0$ as $n\to\infty$.
To prove this statement the following intermediate step is required;
	we recall the $\alpha\bc\nix$ notation from \Sec~\ref{Sec_firstMoment}.
The $d<2(k-1)\ln(k-1)$ case of \Lem~\ref{Lemma_Z2} was previously proved in~\cite{Prasad} by way of the second moment analysis from~\cite{AchNaor}.
As it turns out, the regime $2(k-1)\ln(k-1)<d<\dc$ requires a somewhat more sophisticated argument.
In any case, for the sake of completeness we give a full prove of \Lem~\ref{Lemma_Z2}, including the $d<2(k-1)\ln(k-1)$
	(which adds merely three lines to the argument).
Similarly, in~\cite{Prasad} the following claim was established in the case $d<2(k-1)\ln(k-1)$.

\begin{claim}\label{Cor_phiFirstMoment}
Suppose that $d<\dc$ and that $\omega=\omega(n)$ is such that $\lim_{n\to\infty}\omega(n)=\infty$ but $\omega=o(n)$.
Then \whp\ $\G$ is such that
	$$\bck{\vecone\cbc{\norm{\alpha(\SIGMA)-\bar\alpha}_2>\sqrt{\omega/n}}}_{\G}\leq\exp(-\Omega(\omega)).$$
\end{claim}
\begin{proof}
We combine \Thm~\ref{Thm_Z} with a standard ``first moment'' estimate similar to the proof of~\cite[\Lem~5.4]{Prasad}.
The entropy function $\alpha\in\cP(\brk k)\mapsto H(\alpha)=-\sum_{i=1}^k\alpha_i\ln\alpha_i$ is  concave and attains its global maximum at $\bar\alpha$.
In fact, the Hessian of $\alpha\mapsto H(\alpha)$ satisfies $D^2H(\alpha)\preceq-2\id$.
Moreover, since $\alpha\mapsto\norm\alpha_2^2$ is convex, $\alpha\mapsto\frac{d}2\ln(1-\norm\alpha_2^2)$ is concave and attains
is global maximum at $\bar\alpha$ as well.
Hence, letting $\varphi$ denote the function from \Lem~\ref{Lemma_phiFirstMoment}, we find $D^2\varphi(\alpha)\preceq-2\id$.
Therefore, we obtain from \Lem~\ref{Lemma_phiFirstMoment} that
	\begin{equation}\label{eqCor_phiFirstMoment1}
	\Erw[Z_\alpha(\G)]\leq \exp(n(\varphi(\bar\alpha)-\norm{\alpha-\bar\alpha}_2^2))\cdot\begin{cases}
		O(1)&\mbox{ if }\norm{\alpha-\bar\alpha}_2> 1/\ln n,\\
		O(n^{(1-k)/2})&\mbox{ otherwise.}
		\end{cases}
	\end{equation}
Further, letting $$Z'(\G)=\sum_{\alpha\in\cP_n(\brk k):\norm{\alpha-\bar\alpha}_2>\sqrt{\omega/n}}Z_\alpha(\G)$$
and treating the cases $\omega\leq \ln^2 n$ and $\omega\geq\ln^2 n$ separetely, we obtain from~(\ref{eqCor_phiFirstMoment1}) that
	\begin{equation}\label{eqCor_phiFirstMoment2}
	\Erw[Z'(\G)]\leq \exp(-\Omega(\omega))\exp(n(\varphi(\bar\alpha)).
	\end{equation}
Since \Lem~\ref{Lemma_firstMoment} shows that $\Erw[\Zkc(\G)]=\Theta(k^n(1-1/k)^m)=\exp(n\varphi(\bar\alpha))$,
(\ref{eqCor_phiFirstMoment2}) yields $\Erw[Z'(\G)]=\exp(-\Omega(\omega))\Erw[\Zkc(\G)]$.
Hence, by Markov's inequality
	\begin{equation}\label{eqCor_phiFirstMoment3}
	\pr\brk{Z'(\G)\leq\exp(-\Omega(\omega))\Erw[\Zkc(\G)]}\geq1-\exp(-\Omega(\omega)).
	\end{equation}
Finally, since $\bck{\norm{\alpha(\SIGMA)-\bar\alpha}_2>\sqrt{\omega/n}}_{\G}=Z'(\G)/\Zkc(\G)$
and because $\Zkc(\G)\geq\Erw[\Zkc]/\omega$ \whp\ by \Thm~\ref{Thm_Z},
 the assertion follows from~(\ref{eqCor_phiFirstMoment3}).
\end{proof}

\begin{proof}[Proof of \Lem~\ref{Lemma_Z2}]
We bound
	$$\Lambda=\sum_{\sigma_1,\sigma_2\in\cS_k(\Gnm)}\norm{\rho(\sigma_1,\sigma_2)-\bar\rho}_2
		=\Zkc(\Gnm)^2\bck{\norm{\rho(\SIGMA_1,\SIGMA_2)-\bar\rho}_2}_{\Gnm}$$
by a sum of three different terms.
First, letting, say, $\omega(n)=\ln n$, we set
	\begin{align*}
	\Lambda_1&=
		\sum_{\sigma_1,\sigma_2\in\cS_k(\Gnm)}\vecone\cbc{\norm{\alpha(\sigma_1)-\bar\alpha}_2>\sqrt{\omega/n}}
			=\Zkc(\Gnm)^2\bck{\norm{\alpha(\SIGMA)-\bar\alpha}_2>\sqrt{\omega/n}}_{\Gnm}.
	\end{align*}
To define the other two, let $\cS_k'(\Gnm)$ be the set of all $\sigma\in\cS_k(\Gnm)$ such that $\norm{\alpha(\sigma)-\bar\alpha}_2\leq\sqrt{\omega/n}$.
Let $\eta>0$ be a small but $n$-independent number and let
	\begin{align*}
	\Lambda_2&=
		\sum_{\sigma_1,\sigma_2\in\cS_k'(\Gnm)}\vecone\cbc{\norm{\rho(\sigma_1,\sigma_2)-\bar\rho}_2\leq\eta}
				\norm{\rho(\sigma_1,\sigma_2)-\bar\rho}_2,&
	\Lambda_3&=		\sum_{\sigma_1,\sigma_2\in\cS_k'(\Gnm)}\vecone\cbc{\norm{\rho(\sigma_1,\sigma_2)-\bar\rho}_2>\eta}.
	\end{align*}
Since $\norm{\rho(\sigma_1,\sigma_2)-\bar\rho}_2\leq2$ for all $\sigma_1,\sigma_2$, we have
	\begin{equation}\label{eqLambda_dec}
	\Lambda\leq4(\Lambda_1+\Lambda_2)+\Lambda_3.
	\end{equation}

Hence, we need to bound $\Lambda_1,\Lambda_2,\Lambda_3$.
With respect to $\Lambda_1$, Claim~\ref{Cor_phiFirstMoment} implies that
	\begin{align}\label{eqLemma_Z2_2}
	\pr\brk{\Lambda_1\leq\exp(-\Omega(\sqrt n))\Zkc(\Gnm)^2}&=1-o(1).
	\end{align}
To estimate $\Lambda_2$,  we let $f$ denote the function from \Lem~\ref{Lemma_f}.
Observe that $Df(\bar\rho)=0$, because $\bar\rho$  maximises the entropy and minimises the $\ell_2$-norm.
Further, a straightforward calculation reveals that for any $i,j,i',j'\in[k],\ (i,j)\neq(i',j')$,
	\begin{align*}
	\frac{\partial^2 f(\rho)}{\partial\rho_{ij}^2}&=-\frac1{\rho_{ij}}+\frac d{1-2/k+\norm{\rho}_2^2}-\frac{2d\rho_{ij}^2}{(1-2/k+\norm{\rho}_2^2)^2},&
	\frac{\partial^2 f(\rho)}{\partial\rho_{ij}\partial\rho_{i'j'}}&=-\frac{2d\rho_{ij}\rho_{i'j'}}{(1-2/k+\norm{\rho}_2^2)^2}.
	\end{align*}
Consequenctly, choosing, say, $\eta<k^{-4}$,  ensures that the Hessian satisfies
	\begin{equation}\label{eqLemma_Z2_11}
	D^2f(\rho)\preceq-2\id\qquad\mbox{ for all $\rho$ such that $\norm{\rho-\bar\rho}_2^2\leq\eta$.}
	\end{equation}
Therefore, \Lem~\ref{Lemma_f}
	yields
	\begin{align}
	\Erw[\Lambda_2]&\leq
		\sum_{\rho\in\cR_{n,k}(\eta)}\norm{\rho-\bar\rho}_2\Erw[Z_\rho^\tensor(\Gnm)]\nonumber\\
		&\leq O(n^{(1-k^2)/2})\exp(nf(\bar\rho))
			\sum_{\rho\in\cR_{n,k}(\eta)}\norm{\rho-\bar\rho}_2\exp(n(f(\rho)-f(\bar\rho)))\nonumber\\
		&\leq O(n^{(1-k^2)/2})\exp(nf(\bar\rho))\sum_{\rho\in\cR_{n,k}(\eta)}\norm{\rho-\bar\rho}_2\exp(-nk^{-2}\norm{\rho-\bar\rho}^2)
				&\mbox{[by~(\ref{eqLemma_Z2_11})]}.
				\label{eqLemma_Z2_11_a}
	\end{align}
Further, since $\rho_{kk}=1-\sum_{(i,j)\neq(k,k)}\rho_{ij}$ for any $\rho\in\cR_{n,k}(\eta)$, substituting $x=\sqrt n\rho$ in~(\ref{eqLemma_Z2_11_a}) yields
	\begin{align}			\label{eqLemma_Z2_11_b}
	\Erw[\Lambda_2]&\leq O(n^{(1-k^2)/2})\exp(nf(\bar\rho)) \int_{\RR^{k^2-1}}\frac{\norm{x}_2}{\sqrt n}\exp(-k^{-2}\norm{x}_2^2)dx
		=O(n^{-1/2})\exp(nf(\bar\rho)).
	\end{align}
Since $f(\bar\rho)=2\ln k+d\ln(1-1/k)$, \Lem~\ref{Lemma_firstMoment} yields
	\begin{align}\label{eqfbarrho}
	\exp(n f(\bar\rho))&\leq O(\Erw[\Zkc(\Gnm)]^2).
	\end{align}
Therefore,  (\ref{eqLemma_Z2_11_b}) entails that
	\begin{align}\label{eqLambda2}
	\Erw[\Lambda_2]&\leq O(n^{-1/2})\Erw[\Zkc(\Gnm)]^2.
	\end{align}

To bound $\Lambda_3$, we consider two separate cases.
The first case is that $d\leq2(k-1)\ln(k-1)$.
Then \Lem~\ref{Lemma_f} and~(\ref{eqfbarrho}) yield
	\begin{equation}\label{eqLambda3_case1}
	\Erw[\Lambda_3]\leq\exp(nf(\bar\rho)-\Omega(n))\leq\exp(-\Omega(n))\Erw[\Zkc(\Gnm)]^2.
	\end{equation}
The second case is that $2(k-1)\ln(k-1)\leq d<\dc$.
We introduce
	\begin{align*}
	\Lambda_{31}&=\sum_{\sigma_1,\sigma_2\in\cS_k'(\Gnm)}\vecone\cbc{\sigma_1\mbox{ fails to be separable}},\\
	\Lambda_{32}&=\sum_{\sigma_1,\sigma_2\in\cS_k'(\Gnm)}\vecone\cbc{\rho(\sigma_1,\sigma_2)\mbox{ is $s$-stable for some $1\leq s\leq k$}},\\
	\Lambda_{33}&=\sum_{\sigma_1,\sigma_2}\vecone\cbc{\rho(\sigma_1,\sigma_2)\mbox{ is $0$-stable and }
		\norm{\rho(\sigma_1,\sigma_2)-\bar\rho}_2>\eta},\\
	\Lambda_{34}&=\sum_{\sigma_1,\sigma_2\in\cS_k'(\Gnm)}\vecone\cbc{\rho(\sigma_1,\sigma_2)\mbox{ is $k$-stable}},
	\end{align*}
so that
	\begin{equation}\label{eqLemma_Z2_1}
	\Lambda_3\leq\Lambda_{31}+\Lambda_{32}+\Lambda_{33}+\Lambda_{34}.
	\end{equation}
By the first part of \Lem~\ref{Lemma_Danny} and Markov's inequality,
	\begin{align}\label{eqLemma_Z2_Lambda31}
	\pr\brk{\Lambda_{31}\leq\exp(-\Omega(n))\Zkc(\Gnm)\Erw[\Zkc(\Gnm)]}&=1-o(1).
	\end{align}
Further, combining \Lem~\ref{Lemma_f} with the second part of \Lem~\ref{Lemma_Danny}, we obtain
	\begin{align}\label{eqLemma_Z2_Lambda32}
	\pr\brk{\Lambda_{32}\leq\exp(n f(\bar\rho)-\Omega(n))}&=1-o(1).
	\end{align}	
Addionally, \Lem~\ref{Lemma_f} and the third part of \Lem~\ref{Lemma_Danny} yield
	\begin{align}\label{eqLemma_Z2_Lambda33}
	\pr\brk{\Lambda_{33}\leq\exp(n f(\bar\rho)-\Omega(n))}&=1-o(1).
	\end{align}	
Moreover, \Lem~\ref{Lemma_clusterSize} entails that
	\begin{align}\label{eqLemma_Z2_Lambda34}
	\pr\brk{\Lambda_{34}\leq\exp(-\Omega(n))\Zkc(\Gnm)\Erw[\Zkc(\Gnm)]}&=1-o(1).
	\end{align}	
Finally, combining (\ref{eqLemma_Z2_Lambda31})--(\ref{eqLemma_Z2_Lambda34}) with (\ref{eqfbarrho}) and~(\ref{eqLemma_Z2_1})
and using Markov's inequality once more, we obtain
	\begin{align}\label{eqLemma_Z2_Lambda3_case2}
	\pr\brk{\Lambda_{3}\leq\exp(-\Omega(n))\Erw[\Zkc(\Gnm)]^2}&=1-o(1).
	\end{align}	
	
In summary, combining (\ref{eqLambda_dec}), (\ref{eqLemma_Z2_2}), (\ref{eqLambda2}), (\ref{eqLambda3_case1}) and~(\ref{eqLemma_Z2_Lambda3_case2})
and setting, say, $\omega=\omega(n)=\ln\ln n$,
we find that  
	\begin{align}\label{eqLemma_Z2_667}
	\pr\brk{\Lambda\leq \sqrt{\omega/n}\,\Erw[\Zkc(\Gnm)]^2}&=1-o(1).
	\end{align}
Since $\Lambda=\Zkc(\Gnm)^2\bck{\norm{\rho(\SIGMA_1,\SIGMA_2)-\bar\rho}_2}_{\Gnm}$
and as $\Zkc(\Gnm)\geq \Erw[\Zkc(\Gnm)]/\omega$ \whp\ by \Thm~\ref{Thm_Z},
the assertion follows 
from~(\ref{eqLemma_Z2_667}).
\end{proof}

\Lem~\ref{Lemma_Z2} puts us in a position to prove \Prop~\ref{Thm_contig} by
extending the argument that was used to ``plant'' single $k$-colorings in~\cite[\Sec~2]{Silent} to the current setting of ``planting'' pairs of $k$-colorings.

\begin{proof}[Proof of \Prop~\ref{Thm_contig}]
Assume for contradiction that $(\cA_n')_{n\geq1}$ is a sequence of events such that for some fixed number $\eps>0$ we have
	\beq\label{eqThm_cont0'}
	\lim_{n\ra\infty}\plp\brk{\cA_n'}=0\quad\mbox{while}\quad\limsup_{n\ra\infty}\prcp\brk{\cA_n'}>2\eps.
	\eeq
Let
	$\omega(n)=\ln\ln1/\plp\brk{\cA_n'}.$
Then $\omega=\omega(n)\to\infty$.
Let $\cB_n$ be the set of all pairs $(\sigma_1,\sigma_2)$ of maps $[n]\to[k]$ such that
	$\norm{\rho(\sigma_1,\sigma_2)-\bar\rho}_2\leq\sqrt{\omega/n}$
and define $$\cA_n=\cbc{(G,\sigma_1,\sigma_2)\in\cA_n':(\sigma_1,\sigma_2)\in\cB_n}.$$
Then \Lem~\ref{Lemma_Z2} and~(\ref{eqThm_cont0'}) imply that
	\begin{equation}\label{eqThm_cont_overlap}
	\lim_{n\ra\infty}\plp\brk{\cA_n}=0\quad\mbox{while}\quad\limsup_{n\ra\infty}\prcp\brk{\cA_n}>\eps.
	\end{equation}
Furthermore,
	\begin{equation}\label{eq_planting0}
	\omega(n)\sim\ln\ln\bc{1/\plp\brk{\cA_n}}\to\infty.
	\end{equation}

For $\sigma_1,\sigma_2:\brk n\to\brk k$
let $\G(n,m|\sigma_1,\sigma_2)$ be the random graph $\gnm$ conditional on the event that $\sigma_1,\sigma_2$ are $k$-colorings.
That is, $\G(n,m|\sigma_1,\sigma_2)$ consists of $m$ random edges that are bichromatic under $\sigma_1,\sigma_2$.
Then 
	\begin{eqnarray}
	\Erw[\Zkc(\gnm)^2\vecone\cbc{\cA_n}]&=&
			\sum_{(\sigma_1,\sigma_2)\in\cB_n}\pr\brk{\sigma_1,\sigma_2\in\cS_k(\gnm), (\gnm,\sigma_1,\sigma_2)\in\cA_n}\nonumber\\
		&=&\sum_{(\sigma_1,\sigma_2)\in\cB_n}\pr\brk{(\gnm,\sigma_1,\sigma_2)\in\cA_n|\sigma_1,\sigma_2\in\cS_k(\gnm)}
				\pr\brk{\sigma_1,\sigma_2\in\cS_k(\gnm)}\nonumber\\
		&=&\sum_{(\sigma_1,\sigma_2)\in\cB_n}
				\pr\brk{\G(n,m|\sigma_1,\sigma_2)\in\cA_n}\cdot\pr\brk{\sigma_1,\sigma_2\in\cS_k(\gnm)}.\label{eqThm_cont1}
	\end{eqnarray}
Letting
	$q_n=\max\cbc{\pr\brk{\sigma_1,\sigma_2\in\cS_k(\gnm)}:(\sigma_1,\sigma_2)\in\cB_n}$,
we obtain from~(\ref{eqThm_cont1}) and the definition {\bf PR1--PR2} of the planted replica model that
	\begin{eqnarray}			\label{eqThm_cont2}
	\Erw[\Zkc(\gnm)^2\vecone\cbc{\cA_n}]
		&\leq&q_n\sum_{(\sigma_1,\sigma_2)\in\cB_n}
				\pr\brk{\G(n,m|\sigma_1,\sigma_2)\in\cA_n}
		\leq k^{2n}q_n\plp\brk{\cA_n}.
	\end{eqnarray}
Furthermore, 
since $\norm{\rho_{i\nix}(\sigma_1,\sigma_2)}_2^2,\norm{\rho_{\nix i}(\sigma_1,\sigma_2)}_2^2\geq1/k$ for all $i\in[k]$,
(\ref{eqErwZrho}) implies
	\begin{align*}
	\frac1n\ln\pr\brk{\sigma_1,\sigma_2\in\cS_k(\gnm)}&\leq
			\frac d2\ln\bc{1-\frac2k+\norm{\rho(\sigma_1,\sigma_2)}_2^2}+O(1/n)\\
				&=d\ln(1-1/k)+O(\omega/n)&\mbox{ for all $(\sigma_1,\sigma_2)\in\cB_n$}.
	\end{align*}
Hence, $q_n\leq(1-1/k)^{2m}\exp(O(\omega))$.
Plugging this bound into~(\ref{eqThm_cont2}) and setting $\bar z=\Erw[\Zkc(\gnm)]$, we see that
	\begin{eqnarray}	\label{eq_planting1}
	\Erw[\Zkc(\gnm)^2\vecone\cbc{\cA_n}]&\leq&
		k^{2n}(1-1/k)^{2m}\exp(O(\omega))\plp\brk{\cA_n}
			=\bar z^2\exp(O(\omega))\plp\brk{\cA_n}.
	\end{eqnarray}

On the other hand, if $\prcp\brk{\cA_n}>\eps$, then 
\Thm~\ref{Thm_Z} implies that
	 $$\prcp\brk{\cA_n\cap\cbc{\Zkc(\gnm)\geq\bar z/\omega}}>\eps/2.$$
Hence, (\ref{eqrr}) yields
	\begin{equation}\label{eq_planting2}
	\Erw[\Zkc(\gnm)^2\vecone\cbc{\cA_n}]\geq\frac{\eps}2\bcfr{\bar z}{\omega}^2.
	\end{equation}
But due to~(\ref{eq_planting0}), (\ref{eq_planting2}) contradicts~(\ref{eq_planting1}).
\end{proof}

\section{Analysis of the planted replica model}\label{Sec_Nor}

\noindent
\emph{In this section we assume that $k\geq3$ and that $d>0$.}

\medskip\noindent
\Prop~\ref{Thm_contig} reduces the task of studying the random replica model to that of analysing the planted replica model,
which we attend to in the present section.
If  $\theta$ is a rooted tree, $\tau_1,\tau_2\in\cS_k(\theta)$, $\omega\geq0$ and if $G$ is a $k$-colorable graph and $\sigma_1,\sigma_2\in\cS_k(G)$, then we let
	$$Q_{\theta,\tau_1,\tau_2,\omega}(G,\sigma_1,\sigma_2)=\frac1n\sum_{v\in[n]}
		\vecone\cbc{\nbh{G}{v}{\sigma_1}\ism(\theta,\tau_1)}\cdot\vecone\cbc{\nbh{G}{v}{\sigma_2}\ism(\theta,\tau_2)}.$$
Additionally, set
	$$q_{\theta,\omega}=\Zkc(\theta)^{-2}\pr\brk{\partial^\omega\T(d)\ism \theta}.$$
The aim in this section is to prove the following statement.

\begin{proposition}\label{Prop_Nor}
Let  $\theta$ be a rooted tree, $\tau_1,\tau_2\in\cS_k(\theta)$ and $\omega\geq0$.
Let $\plG,\plSIGMA_1,\plSIGMA_2$ be chosen from the distribution $\plp$.
Then $Q_{\theta,\tau_1,\tau_2,\omega}(\hat\G,\hat\SIGMA_1,\hat\SIGMA_2)$ converges to $q_{\theta,\omega}$ in probability.
\end{proposition}

Intuitively, \Prop~\ref{Prop_Nor} asserts that in the planted replica model, the distribution of the ``dicoloring'' that $\hat\SIGMA_1,\hat\SIGMA_2$ induce in the
depth-$\omega$ neighborhood of a random vertex $v$ converges to the uniform distribution on the tree that the depth-$\omega$ neighborhood of $v$ induces.
The proof of \Prop~\ref{Prop_Nor}
is by extension of an argument from~\cite{Cond} for the ``standard'' planted model (with a single coloring) to the planted replica model.
More specifically, it is going to be convenient to work with the following {\em binomial} version $\bplp$ of the planted replica model,
where $p\in(0,1)$.
\begin{description}
\item[PR1'] sample two maps $\plSIGMA_1,\plSIGMA_2:[n]\to[k]$ independently and uniformly at random.
\item[PR2'] generate a random graph $\tilde\G$ by including each of the  $\binom{n}{2}-\cF(\plSIGMA_1,\plSIGMA_2)$
	 edges that are bichromatic under both $\plSIGMA_1,\plSIGMA_2$  with probability $p$  independently.
\end{description}
The distributions $\plp$, $\bplp$ are related as follows.

\begin{lemma}\label{Lemma_binmodel}
Let $p=m/\left(\binom{n}{2}(1-1/k)^2\right)$.
For any event $\cE$ we have $\plp\brk\cE\leq O(\sqrt n)\bplp\brk\cE+o(1).$
\end{lemma}
\begin{proof}
Let $\cB$ be the event that $\norm{\rho(\plSIGMA_1,\plSIGMA_2)-\bar\rho}_2^2\leq n^{-1}\ln\ln n$.
Since $\plSIGMA_1,\plSIGMA_2$ are chosen uniformly and independently, the Chernoff bound yields
	\begin{equation}\label{eqLemma_binmodel1}
	\bplp\brk{\cB},\plp\brk{\cB}=1-o(1).
	\end{equation}
Furthermore, given that $\cB$ occurs we obtain $\cF(\plSIGMA_1,\plSIGMA_2)=(2/k-1/k^2)\bink n2+o(n^{3/2})$.
Therefore,  Stirling's formula implies that the event $\cA$ that the graph $\tilde\G$ has precisely $m$ edges satisfies
	\begin{equation}\label{eqLemma_binmodel2}
	\bplp\brk{\cA|\cB}=\Omega(n^{-1/2}).
	\end{equation}
By construction, the binomial model $\bplp$ given $\cA\cap\cB$ is identical to $\plp$ given $\cB$.
Consequently, (\ref{eqLemma_binmodel1}) and~(\ref{eqLemma_binmodel2}) yield
	\begin{align*}
	\plp\brk\cE&\leq\plp\brk{\cE|\cB}+o(1)=\bplp\brk{\cE|\cA,\cB}+o(1)\leq O(\sqrt n)\bplp\brk{\cE}+o(1),
	\end{align*}
as desired.
\end{proof}

The following proofs are based on a simple observation.
Given the colorings $\hat\SIGMA_1,\hat\SIGMA_2$, we can construct $\tilde\G$ as follows.
First, we simply insert each of the $\bink n2$ edges of the complete graph on $\brk n$ with probability $p$ independently.
The result of this is, clearly, the \Erdos-\Renyi\ random graph $\G(n,p)$.
Then, we ``reject'' (i.e., remove) each edge of this graph that joins two vertices that have the same color under either
 $\hat\SIGMA_1$ or $\hat\SIGMA_2$.

\begin{lemma}\label{Lemma_shortcycles}
Let $\omega=\lceil\ln\ln n\rceil$ and assume that $p=O(1/n)$.
\begin{enumerate}
\item Let $\cK(G)$ be the total number of vertices $v$ of the graph $G$ 
	such that $\partial^\omega(G,v)$ contains a cycle.
	Then $$\bplp\brk{\cK(\tilde\G)>n^{2/3}}=o(n^{-1/2}).$$
\item Let $\cL$ be the event that there is a vertex $v$ such that $\partial^\omega(\tilde\G,v)$ contains more than $n^{0.1}$ vertices.
	Then $$\bplp\brk{\cL}\leq\exp(-\Omega(\ln^2n)).$$
\end{enumerate}
\end{lemma}
\begin{proof}
Obtain the random graph $\G'$ from $\tilde\G$ by adding every edge that is monochromatic under either $\plSIGMA_1,\plSIGMA_2$
with probability $p=m/\left(\binom{n}{2}(1-1/k)^2\right)$ independently.
Then $\G'$ has the same distribution as the standard binomial random graph $\G(n,p)$.
Since $\cK(\tilde\G)\leq\cK(\G')$, the first assertion follows from the well-known fact that $\Erw[\cK(\G(n,p))]\leq n^{o(1)}$ and Markov's inequality.
A similar argument yields the second assertion.
\end{proof}

\begin{lemma}\label{Lemma_conc}
Let  $\theta$ be a rooted tree, let $\tau_1,\tau_2\in\cS_k(\theta)$ and let $\omega\geq0$.
Then
	$$\bplp\brk{\abs{Q_{\theta,\tau_1,\tau_2,\omega}(\tilde\G,\hat\SIGMA_1,\hat\SIGMA_2)-
		\Erw[Q_{\theta,\tau_1,\tau_2,\omega}(\tilde\G,\hat\SIGMA_1,\hat\SIGMA_2)]}>n^{-1/3} }\leq\exp(-\Omega(\ln^2n)).$$
\end{lemma}
\begin{proof}
The proof is based on \Lem~\ref{Lemma_Lutz}.
To apply \Lem~\ref{Lemma_Lutz}, we view $(\tilde\G,\plSIGMA_1,\plSIGMA_2)$ as chosen from a product space $X_2,\ldots,X_N$ with $N=2n$ where
$X_v\in\brk{k}^2$ is uniformly distributed for $v\in[n]$ and
where	$X_{n+v}$ is a $0/1$ vector of length $v-1$ whose components are independent $\Be(p)$ variables
	for $v\in[n]$.
Namely, $X_{v}$ with $v\in[n]$ represents the color pair $(\hat\SIGMA_1(v),\hat\SIGMA_2(v))$,
and $X_{n+v}$ for $v\in[n]$ indicates to which vertices $w<v$ 
with $\hat\SIGMA_1(w)\neq\hat\SIGMA_1(v)$,  $\hat\SIGMA_2(w)\neq\hat\SIGMA_2(v)$ vertex $v$ is adjacent (``vertex exposure'').

Define a random variables $S_v=S_v(\tilde\G,\plSIGMA_1,\plSIGMA_2)$ and $S$ by letting 
	\begin{align*}
	S_v &=\vecone\cbc{\nbh{\tilde\G}{v}{\plSIGMA_1}\ism(\theta,\tau_1)}\cdot\vecone\cbc{\nbh{\tilde\G}{v}{\plSIGMA_2}\ism(\theta,\tau_2)},&
	S&=\frac1n\sum_{v\in[n]}S_v.
	\end{align*}
Then 
	\begin{equation}\label{eqQS}
	Q_{\theta,\tau_1,\tau_2,\omega}=S.
	\end{equation}
Further, set $\lambda=n^{0.01}$ and let $\Gamma$ be the event that $|\nbg{\tilde\G}{v}|\leq \lambda$ for all vertices $v$.
	Then by \Lem~\ref{Lemma_shortcycles} we have
	\begin{eqnarray}\label{eqLemma_conc1}
	\pr\brk{\Gamma}&\geq& 1-\exp(-\Omega(\ln^2n)).
	\end{eqnarray}
Furthermore, let $\G'$ be the graph obtained from $\tilde\G$ by removing all edges $e$ that
	are incident with a vertex $v$ such that $|\nbg{\tilde\G}{v}|>\lambda$
	and let 
	\begin{align*}
	S_v'& =\vecone\cbc{\nbh{\G'}{v}{\plSIGMA_2}\ism(\theta,\tau_1)}\cdot\vecone\cbc{\nbh{\G'}{v}{\plSIGMA_2}\ism(\theta,\tau_2)},&
		S'=\frac1n\sum_{v\in[n]}S_v'.
	\end{align*}
If $\Gamma$ occurs, then $S=S'$.
	Hence, (\ref{eqLemma_conc1}) implies that
	\begin{eqnarray}\label{eqLemma_conc2}
	\Erw[S']&=&\Erw[S]+o(1).
	\end{eqnarray}

The random variable $S'$ satisfies~(\ref{eqTL}) with $c=\lambda$ and $c'=n$.
	Indeed, altering either the colors of one vertex $u$ or its set of neighbors can only affect those vertices $v$
	that are at distance at most $\omega$ from $u$, and in $\G'$ there are no more than $\lambda$ such vertices.
	Thus, \Lem~\ref{Lemma_Lutz} applied with, say, $t=n^{2/3}$ and $\gamma=1/n$ and~(\ref{eqLemma_conc1}) yield
	\begin{eqnarray}\label{eqLemma_conc3}
	\pr\brk{|S'-\Erw[S']|>t}\leq\exp(-\Omega(\ln^2n)).
	\end{eqnarray}
Finally, the assertion follows from (\ref{eqQS}), (\ref{eqLemma_conc2}) and~(\ref{eqLemma_conc3}).
\end{proof}

To proceed, we need the following concept.
A {\em $k$-dicolored graph} $(G,v_0,\sigma_1,\sigma_2)$ consists of a $k$-colorable graph $G$ with $V(G)\subset\RR$, a root $v_0\in V(G)$
and two $k$-colorings $\sigma_1,\sigma_2:V(G)\to[k]$.
We call two $k$-dicolored graphs $(G,v_0,\sigma_1,\sigma_2)$, $(G',v_0',\sigma_1',\sigma_2')$ {\em isomorphic}
if there is an isomorphism $\pi:G\to G'$ such that $\pi(v_0)=v_0'$ and $\sigma_1=\sigma_1'\circ\pi$, $\sigma_2=
\sigma_2'\circ\pi$ and such that for any $v,u\in V(G)$ such that $v<u$ we have $\pi(v)<\pi(u)$. 

\begin{lemma}\label{Lemma_bin_0}
	Let  $\theta$ be a rooted tree, let $\tau_1,\tau_2\in\cS_k(\theta)$ and let $\omega\geq0$.
	Then 
	\begin{equation}\label{eqProp_Nor_bin0}
	\Erw\left[Q_{\theta,\tau_1,\tau_2,\omega}(\tilde{\G})\right]=q_{\theta,\omega}+o(1).
	\end{equation}
\end{lemma}
\begin{proof}
Recall that $\T(d)$ is the (possibly infinite) Galton-Watson tree rooted at $v_0$.
Let $\TAU_1,\TAU_2$ denote two $k$-colorings of $\partial^\omega\T(d)$ chosen uniformly at random.
In addition, let $\vec v^*\in[n]$ denote a uniformly random vertex of $\tilde\G$.
To establish~(\ref{eqProp_Nor_bin0}) it suffices to construct a coupling of the random dicolored tree $(\T(d),v_0,\TAU_1,\TAU_2)$
and the random graph $\partial^\omega(\tilde\G,\vec v^*,\hat\SIGMA_1,\hat\SIGMA_2)$ such that
	\begin{align}\label{eqEX0}
	\pr\brk{\partial^\omega(\tilde\G,\vec v^*,\hat\SIGMA_1,\hat\SIGMA_2)\ism(\T(d),v_0,\TAU_1,\TAU_2)}=1-o(1).
	\end{align}
To this end, let $(u(i))_{i\in[n]}$ be a
family of independent random variables such that $u(i)$ is uniformly distributed over the interval $((i-1)/n,i/n)$ for each $i\in[n]$.

The construction of this coupling is based on the principle of deferred decisions.
More specifically, we are going to view the exploration of the depth-$\omega$ neighborhood of $\vec v^*$ in the random graph $\tilde\G$ as a random process,
	reminiscent of the standard breadth-first search process for the exploration of the connected components of the random graph.
The colors of the individual vertices and their neighbors are revealed in the course of the exploration process.
The result of the exploration process will be a dicolored tree $(\hat\T,u(\vec v^*),\hat\TAU_1,\hat\TAU_1)$ whose vertex set is contained in $[0,1]$.
This tree is isomorphic to $\partial^\omega(\tilde\G,\vec v^*,\hat\SIGMA_1,\hat\SIGMA_2)$ \whp\
Furthermore, the distribution of the tree is at total variance distance $o(1)$ from that of $(\T(d),v_0,\TAU_1,\TAU_2)$.

Throughout the exploration process, every vertex is marked either \emph{dead}, \emph{alive}, \emph{rejected} or \emph{unborn}.
The semantics of the marks is similar to the one in the usual ``branching process'' argument for the component exploration in the random graph:
	vertices whose neighbors have been explored are ``dead'', vertices that have been reached but whose neighbors have not yet been inspected are ``alive'',
	and vertices that the process has not yet discovered are ``unborn''.
The additional mark ``rejected'' is necessary because we reveal the colors of the vertices as we explore them.
More specifically, as we explore the neighbors of an alive $v$ vertex,
we insert a ``candidate edge''  between the alive vertex and {\em every} unborn vertex with probability $p$
independently.
If upon revealing the colors of the ``candidate neighbor'' $w$ of $v$ we find a conflict (i.e., $\hat\SIGMA_1(v)=\hat\SIGMA_1(w)$ or
$\hat\SIGMA_2(v)=\hat\SIGMA_2(w)$), we ``reject'' $w$ and the ``candidate edge'' $\{v,w\}$ is discarded.
Additionally, we will maintain for each vertex $v$ a number $D(v)\in[0,\infty]$; the intention is that $D(v)$ is the distance from the root $\vec v^*$ in the part
of the graph that has been explored so far.
The formal description of the process is as follows.
	\begin{description}
		\item[EX1]
		Initially, $\vec v^*$ is alive, $D(\vec v^*)=0$, and all other vertices $v\neq\vec v^*$ are unborn and $D(v)=\infty$.
		Choose a pair of colors $(\hat\SIGMA_1(\vec v^*),\hat\SIGMA_2(\vec v^*))\in[k]^2$ uniformly at random.
		Let $\hat\T$ be the tree consisting of the root vertex $u(\vec v^*)$ only and let 
			$\hat\TAU_h(u(\vec v^*))=\hat\SIGMA_h(\vec v^*)$ for $h=1,2$.
		\item[EX2] 
		While there is an alive vertex $y$ such that $D(y)<\omega$, let $v$ be the  least such vertex.
		For each vertex $w$ that is either rejected or unborn let $a_{vw}=\Be(p)$;
			the random variables $a_{vw}$ are mutually independent.
		For each unborn vertex $w$ such that $a_{vw}=1$ choose a pair $(\hat\SIGMA_1(w),\hat\SIGMA_2(w))\in[k]^2$ independently and uniformly at random
		and set $D(w)=D(v)+1$.
		Extend the tree $\hat\T$ by adding the vertex $u(w)$ and the edge $\{u(v),u(w)\}$ 
			and by setting $\hat\TAU_1(u(w))=\hat\SIGMA_1(w)$, $\hat\TAU_2(u(w))=\hat\SIGMA_2(w)$
			for every unborn $w$ such that $a_{vw}=1$,
				$\hat\SIGMA_1(v)\neq\hat\SIGMA_1(w)$ and $\hat\SIGMA_2(v)\neq\hat\SIGMA_2(w)$.
		Finally, declare the vertex $v$ dead, declare all $w$ with $a_{vw}=1$ and $\hat\SIGMA_1(v)\neq\hat\SIGMA_1(w)$ and
				$\hat\SIGMA_2(v)\neq\hat\SIGMA_2(w)$ alive, and declare all other $w$ with
				$a_{vw}=1$ rejected.
	\end{description}
The process stops once there is no alive vertex $y$ such that $D(y)<\omega$ anymore, at which point we have got a tree $\hat\T$ that is embedded into $[0,1]$.

Let $\cA$ be the event that $\partial^\omega(\hat\G,\vec v^*)$ is an acyclic subgraph that contains no more than $n^{0.1}$ vertices.
Furthermore, let $\cR$ be the event that in {\bf EX2} it never occurs that $a_{vw}=1$ for a rejected vertex $w$.
Then \Lem~\ref{Lemma_shortcycles} implies that $\pr\brk\cA=1-o(1)$.
Moreover, since $p=O(1/n)$ we have $\pr\brk{\cR|\cA}=1-O(n^{-0.8})=1-o(1)$, whence $\pr\brk{\cA\cap\cR}=1-o(1)$.
Further, given that $\cA\cap\cR$ occurs, $\partial^\omega(\hat\G,\vec v^*,\hat\SIGMA_1,\hat\SIGMA_2)$ is isomorphic to
	$(\hat\T,u(\vec v^*),\hat\TAU_1,\hat\TAU_2)$.
Thus,
	\begin{align}\label{eqex1}
	\pr\brk{\partial^\omega(\hat\G,\vec v^*,\hat\SIGMA_1,\hat\SIGMA_2)\ism (\hat\T,u(\vec v^*),\hat\TAU_1,\hat\TAU_2)}=1-o(1).
	\end{align}

Further, if $\cA\cap\cR$ occurs, then whenever {\bf EX2} processes an alive vertex $v$ with $D(v)<\omega$,
the number of unborn neighbors of $v$ of every color combination $(s_1,s_2)$ such that $s_1\neq\hat\SIGMA(v)$, $s_2\neq\hat\SIGMA(v)$
is a binomial random variable whose mean lies in the interval $[np/k^2,(n-n^{0.1})p/k^2]$.
The total variation distance of this binomial distribution and the Poisson distribution $\Po(d/(k-1)^2)$,
which is precisely distribution of the number of children colored $(s_1,s_2)$ in the dicolored Galton-Watson tree,
 is $O(n^{-0.9})$ by the choice of $p$.
In addition, let $\cB$ be the event that each interval $((i-1)/n,i/n)$ for $i=1,\ldots,n$ contains at most one vertex of the tree $\partial^\omega\T(d)$.
Then $\pr\brk{\cB}=1-o(1)$ and
given $\cA\cap\cR$ and $\cB$, there is a coupling of $(\hat\T,u(\vec v^*),\hat\TAU_1,\hat\TAU_2)$
and $\partial^\omega(\T(d),v_0,\TAU_1,\TAU_2)$ such that
	\begin{align}\label{eqex2}
	\pr\brk{\partial^\omega(\T(d),v_0,\TAU_1,\TAU_2)=(\hat\T,u(\vec v^*),\hat\TAU_1,\hat\TAU_2)}=1-o(1).
	\end{align}
Finally, (\ref{eqEX0}) follows from (\ref{eqex1}) and~(\ref{eqex2}).
\end{proof}

\begin{corollary}\label{Lemma_bin}
	Let  $\theta$ be a rooted tree, let $\tau_1,\tau_2\in\cS_k(\theta)$ and let $\omega\geq0$.
	Moreover, let $p= m/(\bink n2(1-1/k)^2)$.
	Then
	\begin{equation}\label{eqProp_Nor1}
	\lim_{\eps\searrow 0}\lim_{n\to\infty}\sqrt n\cdot\bplp\brk{|Q_{\theta,\tau_1,\tau_2,\omega}-q_{\theta,\tau_1,\tau_2,\omega}|>\eps}=0.
	\end{equation}
\end{corollary}
\begin{proof}
This follows by combining \Lem s~\ref{Lemma_conc} and~\ref{Lemma_bin_0}.
\end{proof}

\noindent
Finally, \Prop~\ref{Prop_Nor} is immediate from \Lem~\ref{Lemma_binmodel} and \Cor~\ref{Lemma_bin}.

\section{Establishing local weak convergence}\label{Sec_lwc}

\noindent{\em Throughout this section we assume that $k\geq k_0$ for some large enough constant $k_0$ and that $d<\dc$.}

\medskip
\noindent
Building upon \Prop s~\ref{Thm_contig} and~\ref{Prop_Nor}, we are going to prove \Thm~\ref{Thm_xlwc} and its corollaries.
The key step is to establish the following statement.

\begin{proposition}\label{Prop_replicas}
Let $\omega\geq0$, let $\theta_1,\ldots,\theta_l$ be a rooted trees and let $\tau_1\in\cS_k(\theta_1),\ldots,\tau_l\in\cS_k(\theta_l)$.
Let 
	$$X_n=\sum_{v_1,\ldots,v_l\in[n]}
		\bck{\prod_{i=1}^l\vecone\cbc{\partial^\omega(\G,v_i,\SIGMA)\ism(\theta_i,\tau_i)}}_{\G}.$$
Then $n^{-l}X_n$ converges to $\prod_{i=1}^l\pr\brk{\nb{\T(d)}\ism(\theta_i,\tau_i)}$ in probability.
\end{proposition}

\noindent
The purpose of \Prop s~\ref{Thm_contig} and~\ref{Prop_Nor} was to facilitate the proof of the following fact.

\begin{lemma}\label{Cor_Nor}
Let $\theta$ be a rooted tree and let $\tau\in\cS_k(\theta)$.
Moreover, set
	\begin{align*}
	Q(v)&=\vecone\cbc{\nbg{\G}{v}\ism\theta}\cdot
			\bck{\prod_{j=1}^2\bc{\vecone\cbc{\partial^\omega(\G,v,\SIGMA_j)\ism(\theta,\tau)}-\Zkc(\theta)^{-1}}
		}_{\G},&
	Q&=\frac1n\sum_{v\in[n]}Q(v).
	\end{align*}
Then $Q$ converges to $0$ in probability.
\end{lemma}
\begin{proof}
Let $t(G,v,\sigma)=\vecone\cbc{\nbh{G}{v}{\sigma}\ism(\theta,\tau)}$ and $z=\Zkc(\theta)$ for brevity.
Then
	\begin{align*}
	Q(v)&=\vecone\cbc{\nbg{\G}{v}\ism\theta}\cdot
		\bck{(t(\G,v,\SIGMA_1)-z^{-1})(t(\G,v,\SIGMA_2)-z^{-1})}\\
		&=\vecone\cbc{\nbg{\G}{v}\ism\theta}\cdot
				\bc{\brk{\bck{t(\G,v,\SIGMA_1)t(\G,v,\SIGMA_2)}-z^{-2}}+
					2z^{-1}\brk{z^{-1}-\bck{t(\G,v,\SIGMA)}}}.
	\end{align*}
Hence, setting
	\begin{align*}
	Q'(v)&=\vecone\cbc{\nbg{\G}{v}\ism\theta}\cdot\brk{\bck{t(\G,v,\SIGMA_1)t(\G,v,\SIGMA_2)}-z^{-2}},&
	Q''(v)&=\vecone\cbc{\nbg{\G}{v}\ism\theta}\cdot\brk{z^{-1}-\bck{t(\G,v,\SIGMA)}},\\
	Q'&=\frac1n\sum_{v\in[n]}Q'(v),&
		Q''&=\frac1n\sum_{v\in[n]}Q''(v),
	\end{align*}
we obtain
	\begin{equation}\label{eqQQQ}
	Q=Q'+\frac2zQ''.
	\end{equation}
Now, let $(\hat \G,\hat\SIGMA_1,\hat\SIGMA_2)$ denote a random dicolored graph chosen from the planted replica model and set
	\begin{align*}
	\hat Q'(v)&=\vecone\cbc{\nbg{\hat \G}{v}\ism\theta}\cdot\brk{t(\hat\G,v,\hat\SIGMA_1)t(\hat\G,v,\hat\SIGMA_2)-z^{-2}},&
	\hat Q''(v)&=\vecone\cbc{\nbg{\hat\G}{v}\ism\theta}\cdot\brk{z^{-1}-t(\hat\G,v,\hat\SIGMA_1)},\\
	\hat Q'&=\frac1n\sum_{v\in[n]}\hat Q'(v),&
		\hat Q''&=\frac1n\sum_{v\in[n]}\hat Q''(v),
	\end{align*}
Then \Prop~\ref{Prop_Nor} shows that $\hat Q'$ converges to $0$ in probability.
In addition, applying \Prop~\ref{Prop_Nor} and marginalising $\hat\SIGMA_2$ implies that $\hat Q''$ converges to $0$ in probability as well.
Hence, \Prop~\ref{Thm_contig} entails that $Q'$, $Q''$ converge to $0$ in probability.
Thus, the assertion follows from~(\ref{eqQQQ}).
\end{proof}

\noindent
We complete the proof of \Prop~\ref{Prop_replicas} by generalising the elegant argument that was used in~\cite[\Prop~3.2]{GM}
to establish a statement similar to the $\omega=0$ case of \Prop~\ref{Prop_replicas}.

\begin{lemma}\label{Lemma_replicas}
There exists a sequence $\eps=\eps(n)=o(1)$ such that the following is true.
Let  $\theta_1,\ldots,\theta_l$ be rooted trees, let $\tau_1\in\cS_k(\theta_1),\ldots,\tau_l\in\cS_k(\theta_l)$,
	let $\emptyset\neq J\subset[l]$ and let $\omega\geq0$ be an integer.
For a graph $G$ let  $\cX_{\theta_1,\ldots,\theta_l}(G,J,\omega)$
be the set of all vertex sequences $u_1,\ldots,u_l$ such that $\nbg{G}{u_i}\ism\theta_i$ while
	$$\abs{\bck{\prod_{i\in J}\vecone\cbc{\nbh{G}{u_i}{\SIGMA}\ism(\theta_i,\tau_i)}-\frac1{\Zkc(\theta_i)}}_G}>\eps.$$
Then $|\cX_{\theta_1,\ldots,\theta_l}(\G,J,\omega)|\leq\eps n^l$ \whp
\end{lemma}
\begin{proof}
Let $t_{i}(v,\sigma)=\vecone\cbc{\nbh{\G}{v}{\sigma}\ism(\theta_i,\tau_i)}$ and $z_i=\Zkc(\theta_i)$ for the sake of brevity.
Moreover, set
	\begin{align*}
	Q_i(v)&=\vecone\cbc{\nbg{\G}{v}\ism\theta_i}\cdot\bck{(t_{i}(v,\SIGMA_1)-z_i^{-1})(t_{i}(v,\SIGMA_2)-z_i^{-1})}_{\G},&
	Q_i&=\frac1n\sum_{v\in[n]}Q_i(v).
	\end{align*}
Then \Lem~\ref{Cor_Nor} implies that there exists $\eps=\eps(n)=o(1)$ such that $\sum_{i\in[l]}Q_i\leq\eps^{3}$ \whp\
Therefore, fixing an arbitrary element $i_0\in J$, we see that \whp
	\begin{align*}
	\frac{\eps^2}{n^{l}}|\cX_{\theta_1,\ldots,\theta_l}(\G,J,\omega )|
		&\leq\frac1{n^{l}}\sum_{u_1,\ldots,u_l\in[n]}\bck{\prod_{i\in J}( t_i(u_i,\SIGMA)-z_i^{-1})}_{\hspace{-1mm}\G}^2\,
			\prod_{i=1}^l\vecone\cbc{\nbg{\G}{u_i}\ism\theta_i}\\
		&\hspace{-2cm}\leq
			\frac1{n^{l}}\sum_{u_1,\ldots,u_l\in[n]}\bck{( t_{i_0}(u_{i_0},\SIGMA_1)-z_{i_0}^{-1})( t_{i_0}(u_{i_0},\SIGMA_2)-z_{i_0}^{-1})}_{\hspace{-1mm}\G}\,
			\prod_{i=1}^l\vecone\cbc{\nbg{\G}{u_i}\ism\theta_i}
			&\mbox{[as $\SIGMA_1,\SIGMA_2$ are independent]}\\
		&\hspace{-2cm}\leq
		\frac1{n^{l}}\sum_{u_1,\ldots,u_l\in[n]}Q_{i_0}(u_{i_0})=Q_{i_0}\leq\eps^{3},
	\end{align*}
whence $|\cX_{\theta_1,\ldots,\theta_l}(\G,J,\omega)|\leq\eps n^l$ \whp
\end{proof}

\begin{corollary}\label{Cor_replicas}
Let $\omega\geq0$ be an integer,
let $\theta_1,\ldots,\theta_l$ be rooted trees,
let $\tau_1\in\cS_k(\theta_1),\ldots,\tau_l\in\cS_k(\theta_l)$ and let $\delta>0$.
For a graph $G$ let $Y(G)$ be the number of vertex sequences $v_1,\ldots,v_l$ such that
$\nbg{G}{v_i}\ism\partial^\omega\theta_i$ while
	\begin{equation}\label{eqProp_replicas}
	\abs{\bck{\prod_{i\in[l]}\vecone\cbc{\nbh{G}{v_i}{\SIGMA}\ism(\theta_i,\tau_i)}}_G-
		\prod_{i\in[l]}\frac1{\Zkc(\theta_i)}}>\delta.
	\end{equation}
Then
	$n^{-l}Y(\G)$ converges to $0$ in probability.
\end{corollary}
\begin{proof}
Let $z_i=\Zkc(\partial^\omega\theta_i)$ 
for the sake of brevity.
Let $\cE_{\theta_1,\ldots,\theta_l}$ be the
set of all $l$-tuples $(v_1,\ldots,v_l)$ of distinct vertices such that $\nbg{\G}{v_i}\ism\theta_i$ for all $i\in[l]$.
Moreover, with the notation of \Lem~\ref{Lemma_replicas} let 
	$$\cX_{\theta_1,\ldots,\theta_l}=\bigcup_{\emptyset\neq J\subset[l]}\cX_{\theta_1,\ldots,\theta_l}(\G,J,\omega)$$
and set $\cY_{\theta_1,\ldots,\theta_l}=\cE_{\theta_1,\ldots,\theta_l}\setminus \cX_{\theta_1,\ldots,\theta_l}$.
With $\eps=\eps(n)=o(1)$ from \Lem~\ref{Lemma_replicas},
we are going to show that for each $J\subset[l]$ there exists an ($n$-independent) number $C_J$ such that
	\begin{equation}\label{eqProp_replicas1}
	\abs{\bck{\prod_{i\in J}\vecone\cbc{\nbh{\G}{v_i}{\SIGMA}\ism(\theta_i,\tau_i)}}_{\hspace{-1mm}\G}-
		\prod_{i\in J}z_i^{-1}}\leq C_J\eps^{1/2}
		\qquad\mbox{for all $(v_1,\ldots,v_l)\in\cY_{\theta_1,\ldots,\theta_l}$}.
	\end{equation}
Since $|\cX_{\theta_1,\ldots,\theta_l}|=o(n^l)$ \whp\ by \Lem~\ref{Lemma_replicas}, the assertion follows from~(\ref{eqProp_replicas1}) by setting $J=[l]$.

The proof of (\ref{eqProp_replicas1}) is by induction on $\abs J$.
In the case $J=\emptyset$ there is nothing to show as both products are empty.
As for the inductive step,
set $t_i=\vecone\{\nbh{\G}{v_i}{\SIGMA}\ism(\theta_i,\tau_i)\}$ for the sake of brevity.
Then
	\begin{align}\nonumber
	\bck{\prod_{i\in J}t_i-z_i^{-1}}_{\hspace{-1mm}\G}&=\sum_{I\subset J}(-1)^{|I|}\prod_{i\in I}z_i^{-1}\bck{\prod_{i\in J\setminus I}t_i}_{\hspace{-1mm}\G}\\
		&=\bck{\prod_{i\in J}t_i-\prod_{i\in J}z_i^{-1}}_{\hspace{-1mm}\G}+\prod_{i\in J}z_i^{-1}+\sum_{\emptyset\neq I\subset J}
			(-1)^{|I|}\prod_{i\in I}z_i^{-1}\bck{\prod_{i\in J\setminus I}t_i}_{\hspace{-1mm}\G}.
			\label{eqTelescope}
	\end{align}
By the induction hypothesis, for all $\emptyset\neq I\subset J$ we have
	\begin{equation}\label{eqTelescope2}
	\abs{\bck{\prod_{i\in J\setminus I}t_i}_{\hspace{-1mm}\G}-\prod_{i\in J\setminus I}z_i^{-1}}\leq C_I\eps^{1/2}.
	\end{equation}
Combining (\ref{eqTelescope}) and~(\ref{eqTelescope2}) and using the triangle inequality, we see that there exists $C_J>0$ such that
	\begin{equation}\label{eqTelescope3}
	\abs{\bck{\prod_{i\in J}t_i-z_i^{-1}}_{\hspace{-1mm}\G}-\bck{\prod_{i\in J}t_i-\prod_{i\in J}z_i^{-1}}_{\hspace{-1mm}\G}}\leq C_J\eps^{1/2}/2.
	\end{equation}
Since $(v_1,\ldots,v_l)\not\in\cX_{\theta_1,\ldots,\theta_l}$, we have $\abs{\bck{\prod_{i\in J}t_i-z_i^{-1}}_{\G}}\leq\eps$.
Plugging this bound into~(\ref{eqTelescope3}) yields~(\ref{eqProp_replicas1}).
\end{proof}

\begin{proof}[Proof of \Prop~\ref{Prop_replicas}]
Let $\cU=\cU(\G)$ be the set of all tuples $(v_1,\ldots,v_l)\in[n]^l$ such that
	$\partial^\omega(\G,v_i)\ism\theta_i$ for all $i\in[l]$.
Since the random graph converges locally to the Galton-Watson tree~\cite{BordenaveCaputo}, \whp\ we have 
	\begin{equation}\label{eqProp_replicas_A1}
	|\cU|=o(1)+\prod_{i\in[l]}\pr\brk{\partial^\omega\T(d)\ism\theta_i}
	\end{equation}
(Alternatively, (\ref{eqProp_replicas_A1}) follows from
\Prop s~\ref{Thm_contig} and~\ref{Prop_Nor} by marginalising $\SIGMA_1,\SIGMA_2$.)
The assertion follows by combining~(\ref{eqProp_replicas_A1}) with \Cor~\ref{Cor_replicas}.
\end{proof}

\begin{proof}[Proof of \Thm~\ref{Thm_xlwc}]
As $\cP^2(\cG_k^l)$ carries the weak topology, we need to show that  for any continuous $f:\cP(\cG_k^l)\to\RR$ with a compact support,
	\begin{equation}\label{eqConv0}
	\lim_{n\to\infty}\int f\dd\vec\lambda_{n,m,k}^l=\int f \dd\vec\thet_{d,k}^l.
	\end{equation}
Thus, let $\eps>0$.
Since $\vec\thet_{d,k}^l=\lim_{\omega\to\infty}\vec\thet_{d,k}^l\brk\omega$, we have
	\begin{align*}
	\int f\dd\vec\thet_{d,k}^l&=\lim_{\omega\to\infty}\int f\dd\vec\thet_{d,k}^l[\omega]
		=\lim_{\omega\to\infty}\Erw\int f\dd\atom_{\bigotimes_{i\in[l]}\lambda_{\nb{\T^{i}(d)}}}
		=\lim_{\omega\to\infty}\Erw f\bc{\textstyle\bigotimes_{i\in[l]}\lambda_{\nb{\T^{i}(d)}}}.
	\end{align*}
Hence, there is $\omega_0=\omega_0(\eps)$ such that for $\omega>\omega_0$ we have
	\begin{align}\label{eqConv2} 
	\abs{\int f\dd\vec\thet_{d,k}^l-\Erw f\bc{\textstyle\bigotimes_{i\in[l]}\lambda_{\nb{\T^{i}(d)}}}}<\eps.
	\end{align}
Furthermore, the topology of $\cG_k$ 
is generated by the functions (\ref{eqtopology}). 
Because $f$ has a compact support, this implies that
 there is $\omega_1=\omega_1(\eps)$ such that for any $\omega>\omega_1(\eps)$ and all $\Gamma_1,\ldots,\Gamma_l\in\cG_k$ we have
	\begin{equation}\label{eqOmegaSuffices}
	\abs{f\bc{\bigotimes_{i\in[l]}\atom_{\Gamma_i}}-f\bc{\bigotimes_{i\in[l]}\atom_{\partial^\omega\Gamma_i}}}<\eps.
	\end{equation}
Hence, pick some $\omega>\omega_0+\omega_1$ and assume that $n>n_0(\eps,\omega)$ is large enough.

Let $\vec v_1,\ldots,\vec v_l$ denote vertices of $\G$ that are chosen independently and uniformly at random.
By the linearity of expectation and the definitions of $\vec\lambda_{n,m,k}^l$ and $\lambda_{\G,v_1,\ldots,v_l}$,
	\begin{align*}
	\int f\dd\vec\lambda_{n,d,k}^l&=\Erw\int f\dd\atom_{\lambda_{\G,\vec v_1,\ldots,\vec v_l}}
		=\Erw f(\lambda_{\G,\vec v_1,\ldots,\vec v_l})
	=\Erw\bck{f({\textstyle \bigotimes_{i\in[l]}\atom_{[\G\|\vec v_i,\vec v_i,\SIGMA\|\vec v_i]}})}.
	\end{align*}
Consequently, (\ref{eqOmegaSuffices}) yields
	\begin{align}\label{eqConv1}
	\abs{\int f\dd\vec\lambda_{n,d,k}^l-
			\Erw\bck{f({\textstyle \bigotimes_{i\in[l]}\atom_{\partial^\omega[\G\|\vec v_i,\vec v_i,\SIGMA\|\vec v_i]}})}}
		&<\eps.
	\end{align}
Hence, we need to compare $\Erw\bck{f({\textstyle \bigotimes_{i\in\brk l}\atom_{\partial^\omega[\G\|\vec v_i,\vec v_i,\SIGMA\|\vec v_i]}})}$
and $\Erw f\bc{\textstyle\bigotimes_{i\in[l]}\lambda_{\nb{\T^{i}(d)}}}$.
	
Because the tree structure of $\T(d)$ stems from a Galton-Watson branching process,
there exist a finite number of pairwise non-isomorphic rooted trees $\theta_1,\ldots,\theta_h$ together
with $k$-colorings $\tau_1\in\cS_k(\theta_1),\ldots,\tau_h\in\cS_k(\theta_h)$ such that
with $p_i=\pr\brk{\nb{\T(d)}\ism(\theta_i,\tau_i)}$ we have
	\begin{equation}\label{eqNasty1}
	\sum_{i\in[h]}p_i>1-\eps.
	\end{equation}
Further, \Prop~\ref{Prop_replicas} implies that for $n$ large enough and any $i_1,\ldots,i_l\in[h]$ we have
	\begin{align}\label{eqNasty2}
	\Erw\abs{\bck{\prod_{i=1}^l\vecone\cbc{\partial^\omega[\G\|\vec v_i,\vec v_i,\SIGMA\|\vec v_i]\ism(\theta_{h_i},\tau_{h_i})}}-\prod_{i\in[l]}p_{h_i}}<h^{-l}\eps.
	\end{align}
Combining (\ref{eqOmegaSuffices}), (\ref{eqNasty1}) and (\ref{eqNasty2}), we conclude that
	\begin{equation}\label{eqNasty3}
	\abs{\Erw\bck{f({\textstyle \bigotimes_{i\in\brk l}\atom_{\partial^\omega[\G\|\vec v_i,\vec v_i,\SIGMA\|\vec v_i]}})}-
		\Erw f\bc{\textstyle\bigotimes_{i\in[l]}\lambda_{\nb{\T^{i}(d)}}}}<3l\norm f_\infty\eps.
	\end{equation}
Finally,
(\ref{eqConv0}) follows from (\ref{eqConv2}), (\ref{eqConv1}) and~(\ref{eqNasty3}).
\end{proof}

\begin{proof}[Proof of \Cor~\ref{Cor_xlwc}]
While it is not difficult to derive \Cor~\ref{Cor_xlwc} from \Thm~\ref{Thm_xlwc},
\Cor~\ref{Cor_xlwc} is actually immediate from \Prop~\ref{Prop_replicas}.
\end{proof}

\begin{proof}[Proof of \Cor~\ref{Cor_decay}]
\Cor~\ref{Cor_decay} is simply the special case of setting $\omega=0$ in \Cor~\ref{Cor_xlwc}.
\end{proof}

\begin{proof}[Proof of \Cor~\ref{Cor_reconstr}]

For  integer $\omega\geq 0$, consider the quantities 
$\frac1n\sum_{v\in[n]}\Erw[\corr_{k,\gnm}(v, \omega)]$ and $\Erw[\corr_{k,\partial^{\omega}\T(d)}(v_0, \omega)]$.
The corollary follows by showing that
\begin{eqnarray}\label{eq:target-Cor_reconstr}
\left|\frac1n\sum_{v\in[n]}\Erw[\corr_{k,\gnm}(v, \omega)]- \Erw[\corr_{k,\partial^{\omega}\T(d)}(v_0, \omega)] \right |=o(1).
\end{eqnarray}
Let us call ${\cA}$, the quantity on the l.h.s. of the above equality. It holds that
\begin{eqnarray}
\cA&\leq & \left|\frac1n\sum_{v\in[n]}\left(\Erw[\corr_{k,\gnm}(v, \omega)]- \Erw[\corr_{k,\partial^{\omega}\gnm}(v_0, \omega)] \right )\right | 
+\left|\frac1n\sum_{v\in[n]}\Erw[\corr_{k,\partial^{\omega}\gnm}(v, \omega)]- \Erw[\corr_{k,\partial^{\omega}\T(d)}(v_0, \omega)] \right |. \nonumber
\end{eqnarray}

We observe that, for any $v$-rooted  $G \in \mathfrak G$ and $\omega$ it holds that  $\corr_{k,G}(v,\omega)\in [0,1]$. 
Then,  by using Corollary \ref{Cor_xlwc} where $l=1$ (i.e. weak convergence) we get that
\begin{eqnarray}\label{eq:AbsBound1st}
\left|\frac1n\sum_{v\in[n]}\left(\Erw[\corr_{k,\gnm}(v, \omega)]- \Erw[\corr_{k,\partial^{\omega}\gnm}(v_0, \omega)] \right )\right | =o(1).
\end{eqnarray}
For bounding the second quantity we use the following observation:
The above implies that
\begin{eqnarray}\label{eq:2ndQuantVsTVD}
\left|\frac1n\sum_{v\in[n]}\Erw[\corr_{k,\partial^{\omega}\gnm}(v, \omega)]- \Erw[\corr_{k,\partial^{\omega}\T(d)}(v_0, \omega)] \right |
\leq \pr\brk{\partial^{\omega}(G(n,m),v^*) \not\ism \partial^{\omega}\T(d)}  \cdot \max_{\theta} \{ \corr_{k,\theta}(v, \omega)\}, 
\end{eqnarray}
where the $v^*$ is a randomly chosen vertex of $G(n,m)$. The probability term 
$\pr\brk{\partial^{\omega}(G(n,m),v^*) \not\ism \partial^{\omega}\T(d)}$ is w.r.t. any coupling
of $\partial^{\omega}(G(n,m),v^*)$ and $\partial^{\omega}\T(d)$. Also, the maximum index $\theta$ varies over all trees 
with at most $n$ vertices and  with at most $\omega$ levels.

Working as in Lemma \ref{Lemma_bin_0} we get  the following: There is a coupling
of $\partial^{\omega}(G(n,m),v^*)$ and $\partial^{\omega}\T(d)$, where   $d=2m/n$,  such that 
\begin{eqnarray}
\pr\brk{\partial^{\omega}(G(n,m),v) \ism \partial^{\omega}\T(d)}=1-o(1).\label{eq:AsymEqDistr}
\end{eqnarray}
 Plugging (\ref{eq:AsymEqDistr}) into (\ref{eq:2ndQuantVsTVD}) we get that
\begin{eqnarray}\label{eq:AbsBound2nd}
\left|\frac1n\sum_{v\in[n]}\Erw[\corr_{k,\partial^{\omega}\gnm}(v, \omega)]- \Erw[\corr_{k,\partial^{\omega}\T(d)}(v_0, \omega)] \right |=o(1),
\end{eqnarray}
since it always holds that $\corr_{k,\theta}(v, \omega)\in [0,1]$.
From (\ref{eq:AbsBound1st}) and (\ref{eq:AbsBound2nd}), we get that $\cA=o(1)$, i.e.  (\ref{eq:target-Cor_reconstr})
is true. The corollary follows.

%

%



\end{proof}

\begin{remark}
Alternatively, we could have deduced \Cor~\ref{Cor_reconstr} from \Lem~\ref{Lemma_Z2} and \cite[\Thm~1.4]{GM}.
\end{remark}

\bigskip
\noindent{\bf Acknowledgement.}
We thank Ralph Neininger for helpful discussions.

\begin{appendix}

\section{Convergence of $\vec\thet_{d,k}^l\brk\omega$}\label{Sec_TreeSeq}

\noindent
We use a standard argument to prove that the sequence defined in~(\ref{eqTreeSeq}) converges.

\begin{lemma}\label{Lemma_TreeSeq}
The sequence $(\vec\thet_{d,k}^l\brk\omega)_{\omega\geq1}$ converges for any $d>0,k\geq3,l>0$.
\end{lemma}
\begin{proof}
The space $\cP^2(\cG_k^l)$ is Polish and thus complete.
Therefore, it suffices to prove that $(\vec\thet_{d,k}^l\brk\omega)_{\omega\geq1}$ is a Cauchy sequence.
As $\cP^2(\cG_k^l)$ is endowed with the weak topology,
this amounts to proving that for 
any bounded continuous function $f:\cP(\cG_k^l)\to\RR$ with a compact support and any $\eps>0$
there exists integer $N=N(\eps)\ge 0$ such that
	\begin{equation}\label{eqNor0}
	\left|\int f \dd\vec\thet_{d,k}^l\brk {\omega_1}- \int f\dd\vec\thet_{d,k}^l\brk {\omega_2}\right| < \eps
		\qquad\mbox{if $\omega_1,\omega_2\ge N$.}
	\end{equation}
By the definition of $\vec\thet_{d,k}^l$,
\begin{align}\label{eqNor1}
\int f \dd\vec\thet_{d,k}^l\brk {\omega} = 
	\Erw\int f \dd \atom_{\bigotimes_{i\in[l]}\lambda\bc{\partial^{\omega} \T^{i}(d)}}
		= \Erw f\bc{\textstyle\bigotimes_{i\in[l]}\lambda_{\partial^{\omega} \T^i(d)}}.
\end{align}
Hence, to prove (\ref{eqNor0}) if suffices to show that for any $\eps>0$ there is $N(\eps)>0$ such that
	\begin{equation}\label{eqNor01}
	\Erw\left|
	{
		f\bc{\textstyle\bigotimes_{i\in[l]}\lambda_{\partial^{\omega_1} \T^i(d)}}-f\bc{\textstyle\bigotimes_{i\in[l]}\lambda_{\partial^{\omega_2} \T^i(d)}}}
	\right| < \eps
		\qquad\mbox{for all $\omega_1,\omega_2\ge N$.}
	\end{equation}

To establish~(\ref{eqNor01}), we observe that the sequence $\lim_{\omega\to\infty}\lambda_{\partial^\omega T}$
converges for any locally finite rooted tree $T$.
Indeed, $(\lambda_{\partial^\omega T})_\omega$ is a sequence in the space $\cP(\cG_k)$, which, equipped with the weak topology, is Polish.
Hence, it suffices to prove that for any continuous function $g:\cG_k\to\RR$ with a compact support the sequence
	$\bc{\int g\dd\lambda_{\partial^\omega T}}_{\omega}$ converges.
Indeed, because the topology of $\cG_k$ is generated by the functions of the form~(\ref{eqtopology}), it suffices to verify that
that for any $\Gamma\in\cG_k$ and any $\omega_0\geq0$ the 
	sequence $\bc{\int g_{\Gamma,\omega_0}\dd\lambda_{\partial^\omega T}}_{\omega}$ converges, where
	$$g_{\Gamma,\omega_0}:\cG_k\to\cbc{0,1},\qquad \Gamma'\mapsto\vecone\cbc{\partial^{\omega_0}\Gamma=\partial^{\omega_0}\Gamma'}.$$
But this last convergence statement holds simply because the construction of $\lambda_{\partial^\omega T}$ ensures that
	$$\int g_{\Gamma,\omega_0}\dd\lambda_{\partial^\omega T}=\int g_{\Gamma,\omega_0}\dd\lambda_{\partial^{\omega_0} T}
		\qquad\mbox{for all }\omega>\omega_0.$$

Finally, because $\lim_{\omega\to\infty}\lambda_{\partial^\omega T}$ exists for any $T$,
	(\ref{eqNor01}) follows from the fact that the continuous function $f$ has a compact support.
\end{proof}

\end{appendix}

\end{document}